\documentclass{amsart}

\usepackage{amssymb}
\usepackage{amsfonts}
\usepackage{amsmath}
\usepackage{amsthm}
\usepackage{graphicx}
\usepackage{mathrsfs}
\usepackage{dsfont}
\usepackage{amscd}
\usepackage{multirow}
\usepackage[all]{xy}
\normalfont
\usepackage[T1]{fontenc}
\usepackage{calligra}
\usepackage{verbatim}
\usepackage[usenames,dvipsnames]{color}
\usepackage[colorlinks=true,linkcolor=RoyalBlue,citecolor=Cerulean,pdfauthor={Fr{\'e}d{\'e}ric Latr{\'e}moli{\`e}re, PhD},pdftitle={Metric Approximations of Spectral Triples on the Sierpi{\'n}ski Gasket and other fractal curves}]{hyperref}
\usepackage{enumerate}
\usepackage{slashed}
\usepackage{array}

\newcommand{\N}{{\mathds{N}}}
\newcommand{\Z}{{\mathds{Z}}}

\newcommand{\R}{{\mathds{R}}}
\newcommand{\C}{{\mathds{C}}}
\newcommand{\T}{{\mathds{T}}}

\newcommand{\D}{{\mathfrak{D}}}
\newcommand{\A}{{\mathfrak{A}}}
\newcommand{\B}{{\mathfrak{B}}}

\newcommand{\Nbar}{\overline{\N}}
\newcommand{\Lip}{{\mathsf{L}}}
\newcommand{\Hilbert}{{\mathscr{H}}}

\newcommand{\GH}{{\mathsf{GH}}}

\newcommand{\dpropinquity}[1]{{\mathsf{\Lambda}^\ast_{#1}}}

\newcommand{\modpropinquity}[1]{{\mathsf{\Lambda}^{\mathsf{mod}}_{#1}}}
\newcommand{\dmodpropinquity}[1]{{\mathsf{\Lambda}^{\ast\mathsf{mod}}_{#1}}}
\newcommand{\dmetpropinquity}[1]{{\mathsf{\Lambda}^{\ast\mathsf{met}}_{#1}}}

\newcommand{\spectralpropinquity}{{\mathsf{\Lambda}^{\mathsf{spec}}}}

\newcommand{\Kantorovich}[1]{{\mathsf{mk}_{#1}}}

\newcommand{\Haus}[1]{{\mathsf{Haus}_{#1}}}

\newcommand{\StateSpace}{{\mathscr{S}}}

\newcommand{\MongeKant}{{Mon\-ge--Kan\-to\-ro\-vich metric}}

\newcommand{\gQVB}{metrized quantum vector bundle}

\newcommand{\qvb}[3]{{\mathrm{qvb}\left({#1},{#2},{#3}\right)}}

\newcommand{\qcms}{quantum compact metric space}

\newcommand{\unit}{1}

\newcommand{\sa}[1]{{\mathfrak{sa}\left({#1}\right)}}

\newcommand{\inner}[3]{{\left<{#1},{#2}\right>_{#3}}}

\newcommand{\dom}[1]{{\operatorname*{dom}\left({#1}\right)}}

\newcommand{\norm}[2]{\left\|{#1}\right\|_{#2}}

\newcommand{\CDN}{{\mathsf{DN}}}

\newcommand{\worknote}[1]{}
\newcommand{\opnorm}[3]{{\left|\mkern-1.5mu\left|\mkern-1.5mu\left| {#1} \right|\mkern-1.5mu\right|\mkern-1.5mu\right|_{#3}^{#2}}}

\newcommand{\tunnelmagnitude}[2]{{\mu\left({#1}\middle\vert{#2}\right)}}

\newcommand{\tunnelextent}[1]{{\chi\left({#1}\right)}}

\newcommand{\co}[1]{{\overline{\mathrm{co}}\left(#1\right)}}

\newcommand{\module}[1]{{\mathscr{#1}}}

\newcommand{\spectrum}[1]{\mathrm{Sp}\left({#1}\right)}

\theoremstyle{plain}
\newtheorem{theorem}{Theorem}[section]

\newtheorem{lemma}[theorem]{Lemma}

\newtheorem{theorem-definition}[theorem]{Theorem-Definition}

\theoremstyle{definition}
\newtheorem{definition}[theorem]{Definition}

\newtheorem{convention}[theorem]{Convention}
\newtheorem{hypothesis}[theorem]{Hypothesis}

\theoremstyle{remark}

\newtheorem{remark}[theorem]{Remark}

\newtheorem{notation}[theorem]{Notation}

\renewcommand{\geq}{\geqslant}
\renewcommand{\leq}{\leqslant}

\newcommand{\Siep}{{Sierpi{\'n}ski}}
\newcommand{\SiepG}{{{\Siep} gasket}}
\newcommand{\SG}[1]{\mathcal{SG}_{#1}}
\newcommand{\HG}[1]{\mathcal{HG}_{#1}}
\newcommand{\FC}[1]{\mathcal{FC}_{#1}}
\newcommand{\Dirac}{{\slashed{D}}}
\newcommand{\range}[1]{{\mathrm{range}\left({#1}\right)}}

\numberwithin{equation}{section}

\allowdisplaybreaks[4]

\begin{document}

\title[The Sierpi{\'n}ski triangle as a limit]{Metric Approximations of Spectral Triples on the Sierpi{\'n}ski Gasket and other fractal curves}
\author{Th{e}r{e}se-Marie Landry}
\address{Department of Mathematics, University of California, Riverside, CA 92521, USA.}
\email{tland004@ucr.edu}

\author{Michel L. Lapidus}
\address{Department of Mathematics, University of California, Riverside, CA 92521, USA.}
\email{lapidus@math.ucr.edu}
\urladdr{http://www.math.ucr.edu/~lapidus}
\thanks{The work of the second author (MLL) was partially supported by the US National Science Foundation (NSF) under the research grants DMS-0707524 and DMS-1107750, the Academic Senate of the University of California, Riverside, as well as the Burton Jones Endowed Chair in Pure Mathematics (since 1997) at the University of California, Riverside.}

\author{Fr\'{e}d\'{e}ric Latr\'{e}moli\`{e}re}
\address{Department of Mathematics \\ University of Denver \\ Denver CO 80208, USA}
\email{frederic@math.du.edu}
\urladdr{http://www.math.du.edu/\symbol{126}frederic}

\date{\today}

\begin{abstract}
  Noncommutative geometry provides a framework, via the construction of spectral triples, for the study of the geometry of certain classes of fractals. Many fractals are constructed as natural limits of certain sets with a simpler structure: for instance, the {\SiepG} is the limit of finite graphs consisting of various affine images of an equilateral triangle. It is thus natural to ask whether the spectral triples, constructed on a class of fractals called piecewise $C^1$-fractal curves, are indeed limits, in an appropriate sense, of spectral triples on the approximating sets. We answer this question affirmatively in this paper, where we use the spectral propinquity on the class of metric spectral triples, in order to formalize the sought-after convergence of spectral triples. Our results and methods are relevant to the study of analysis on fractals and have potential physical applications.
\end{abstract}
\maketitle

\tableofcontents




\section{Introduction}

Our purpose is to demonstrate how the spectral triples of \cite{Lapidus08,Lapidus14} constructed on such fractals as the {\SiepG} or the harmonic gasket are indeed the limits, for the spectral propinquity, of spectral triples on finite graphs which naturally approximate these fractals. In short, the spectral triples constructed in \cite{Lapidus08} by Christensen, Ivan and the second author for the {\SiepG}, and then, in \cite{Lapidus14}, by the second author and Sarhad, for a larger class of fractals built with $C^1$ curves, including the harmonic gasket, a typical example of a measurable, fractal Riemannian manifold, can be described as follows. They are infinite direct sums of spectral triples associated with each curve out of which the given fractal is built. Alternatively, and somewhat heuristically, they can be thought of as suitable limits of spectral triples attached to each natural pre-fractal approximation, or ``cell'', of the fractal. The key objective of the present paper is to make precise and to rigorously establish the latter observation, using the \emph{spectral propinquity} introduced by the third author.

\bigskip

The spectral propinquity is a distance on metric spectral triples introduced by the third author as part of a project to devise an analytic framework to discuss approximations of noncommutative geometric structures, from {\qcms s} to metric spectral triples, including Hilbert modules and group actions. Thus, the present work provides new examples of applications of noncommutative metric geometry, and other functional analytic methods in metric geometry, to fractal geometry. The results in this paper are an example of an approximation of an analogue of a differential structure by metric means, taking advantage of the generalization of differential structures given by Connes' spectral triples.

\bigskip

Our journey begins with a new form of geometry made possible by the duality between compact Hausdorff spaces and certain commutative Banach algebras, called commutative C*-algebras. Gelfand defined a (unital) \emph{$C^\ast$-algebra} $\A$ as a unital Banach algebra over $\C$, endowed with an anti-multiplicative, conjugate linear involution $a\in\A\mapsto a^\ast\in\A$ such that $\norm{a^\ast a}{\A} = \norm{a}{\A}^2$, for all $a\in\A$, where $\norm{\cdot}{\A}$ is the norm on $\A$. An example of a commutative, unital C*-algebra is the algebra $C(X)$ of all $\C$-valued continuous functions over a compact Hausdorff space $X$, equipped with the supremum norm. In fact, Gelfand, Naimark and Segal proved that the category of Abelian unital C*-algebras with their natural morphisms, called *-morphisms, is dual to the category of compact Hausdorff spaces with continuous maps. Thus, studying the topology of compact Hausdorff spaces is equivalent to studying commutative, unital C*-algebras.

\bigskip

All C*-algebras are *-isomorphic to a closed subalgebra of operators on a Hilbert space, closed under the adjoint operation. In particular, C*-algebras are the natural algebras for observables in quantum mechanics, or for the description of local observables in quantum field theory. Moreover, via such constructions as C*-crossed-products or groupoid C*-algebras, there are natural ways to associate noncommutative C*-algebras to certain geometric situations where the spaces of interest may be highly singular, such as the orbit space for minimal actions of $\Z$ on the Cantor set. Since the category of Abelian unital C*-algebras is dual to the category of compact Hausdorff spaces, it was then natural to generalize various constructions from topology, such as $K$-theory, to all C*-algebras. This line of research has been very beneficial to both the study of C*-algebras and to topology.

\bigskip

Since C*-algebras enable the study of noncommutative topology, it is natural to ask whether other geometric structures could be extended to noncommutative algebras. Our present work involves two such research directions: noncommutative Riemannian geometry, and noncommutative metric geometry.

\bigskip

Connes proposes that Dirac operators on the $L^2$-sections of the spinor bundle for connected compact Riemannian spin manifolds can be generalized to possibly noncommutative algebras with the structure of \emph{spectral triples}, a form of quantum differential structure. There is some amount of variation in the definition of a spectral triple, but the following definition seems to contain the common traits of these objects, and is a good starting point for the present work.

\begin{definition}[{\cite{Connes89,Connes}; see also, e.g., \cite{sixquarter}}]\label{spectral-triple-def}
  A \emph{spectral triple} $(\A,\Hilbert,D)$ consists of a unital C*-algebra $\A$, a Hilbert space $\Hilbert$ which is a left $\A$-module, and a self-adjoint operator $D$ defined on a dense linear subspace, $\dom{D}$, of $\Hilbert$ such that
  \begin{enumerate}
    \item $D+i$ has a compact inverse,
    \item the set of $a\in\A$ such that
      \begin{equation*}
        a \cdot \dom{D} \subseteq \dom{D} 
      \end{equation*}
      and
      \begin{equation*}
        [D,a] \text{ is closeable, with bounded closure}
      \end{equation*}
      is dense in $\A$.
  \end{enumerate}
\end{definition}

In addition to cohomological information (in the form of a cyclic cocycle), Connes then noted in \cite{Connes89} that a spectral triple $(\A,\Hilbert,D)$ always induces an extended metric on the state space of its underlying C*-algebra $\A$, by setting, for any two states $\varphi,\psi$ of $\A$,
\begin{equation}\label{distance-eq}
  \Kantorovich{D}(\varphi,\psi) = \sup\left\{ \left|\varphi(a) - \psi(a)\right| : a\in \A, a = a^\ast, \opnorm{[D,a]}{}{\Hilbert} \leq 1  \right\}\text{.}
\end{equation}
He also observed in \cite{Connes89} that when $D$ is actually the Dirac operator on a connected compact spin Riemannian manifold, then $\Kantorovich{D}$ restricts on the space of points, i.e. pure states, to the geodesic distance on the manifold.

In fact, the extended metric defined by Expression (\ref{distance-eq}) is, in the classical picture of the Dirac operator on a manifold or for the spectral triples from \cite{Lapidus08,Lapidus14}, a special case of the {\MongeKant}, introduced by Kantorovich \cite{Kantorovich40,Kantorovich58}. Thus spectral triples provide one possible route to define an analogue of the {\MongeKant} within the broader context of noncommutative C*-algebras.

\bigskip

Rieffel \cite{Rieffel98a,Rieffel99,Rieffel00} started the study of noncommutative metric geometry, where the basic objects are noncommutative analogues of the algebras of Lipschitz functions over compact metric spaces, called {\qcms s}. Indeed, if $(X,d)$ is a (compact) metric space, and if $f \in C(X)$, then the \emph{Lipschitz seminorm} of $f$ is defined as
\begin{equation*}
  \Lip_d(f) = \sup\left\{ \frac{|f(x) - f(y)|}{d(x,y)} : x,y \in X, x\neq y \right\} \text,
\end{equation*}
allowing for the value $\infty$.

The {\MongeKant}, introduced by Kantorovich \cite{Kantorovich40,Kantorovich58} in his study of Monge's transportation problem, is defined for any two Radon probability measures $\mu$,$\nu$ on $X$ by
\begin{equation}\label{Kantorovich-eq}
  \Kantorovich{\Lip_d}(\mu,\nu) = \sup\left\{ \left| \int_X f \, d\mu - \int_X f\, d\nu \right| : f\in C(X), \Lip_d(f) \leq 1 \right\}\text.
\end{equation}
It is an important metric, for instance, in probability theory and the study of the transportation problem, in particular because it induces the weak* topology on the space $\StateSpace(C(X))$ of Radon probability measures on $X$. Moreover, the map which sends any point $x$ in $X$ to the Dirac probability measure at $x$ becomes an isometry from $(X,d)$ to $(\StateSpace(C(X)),\Kantorovich{\Lip_d})$ --- thus allowing one to recover the original metric $d$ on $X$.

The \emph{state space} of a possibly noncommutative C*-algebra is defined as the set of its positive linear functionals of norm $1$, and Rieffel used the fundamental properties of the {\MongeKant} as the starting point for the study of {\qcms s}. The actual definition of a {\qcms} has evolved as research on the subject has progressed, and seems to now settle around the following notion.

\begin{definition}[{\cite{Connes89,Rieffel98a,Rieffel99,Latremoliere13}}]\label{qcms-def}
  A \emph{\qcms} $(\A,\Lip)$ is an ordered pair of a unital C*-algebra $\A$ and a seminorm $\Lip$ defined on a dense subspace $\dom{\Lip}$ of the self-adjoint part
  \begin{equation*}
    \sa{\A} = \left\{ a \in \A : a = a^\ast \right\}
  \end{equation*}
  of $\A$ such that the following four conditions hold:
  \begin{enumerate}
  \item $\{a\in\dom{\Lip} : \Lip(a) = 0 \} = \R\unit_\A$,
  \item the \emph{\MongeKant} $\Kantorovich{\Lip}$ defined between any two states $\varphi,\psi \in \StateSpace(\A)$ of $\A$ by
    \begin{equation*}
      \Kantorovich{\Lip}(\varphi,\psi) = \sup\left\{ |\varphi(a) - \psi(a)| : a\in\dom{\Lip}, \Lip(a) \leq 1 \right\}
    \end{equation*}
    metrizes the weak* topology on $\StateSpace(\A)$,
  \item for all $a,b \in \dom{\Lip}$, we have
    \begin{equation*}
      \frac{ab+ba}{2}, \frac{ab-ba}{2i} \in \dom{\Lip}\text{,}
    \end{equation*}
    and $\Lip$ satisfies the \emph{Leibniz inequality}
    \begin{equation*}
      \max\left\{ \Lip\left(\frac{ab+ba}{2}\right), \Lip\left(\frac{ab-ba}{2i}\right) \right\} \leq \Lip(a)\norm{b}{\A} + \norm{a}{\A} \Lip(b)\text{,}
    \end{equation*}
  \item $\{ a\in \dom{\Lip} : \Lip(a) \leq 1 \}$ is closed for the topology induced by $\norm{\cdot}{\A}$ on $\A$.
  \end{enumerate}
\end{definition}

\begin{convention}
  In this paper, by convention, if $\Lip$ is a seminorm defined on a dense subspace $\dom{\Lip}$ of a normed vector space $E$, and if $a\in E\setminus\dom{\Lip}$, then we set $\Lip(a) = \infty$. When using this convention, $0\cdot\infty = 0$, while $x\infty = \infty + x = x+\infty = 0+\infty = \infty + 0 = \infty$ for all $x > 0$. 
\end{convention}

Rieffel's motivation for the introduction of {\qcms s} was, in part, to provide a formal framework for certain approximations of noncommutative algebras found in the literature in quantum physics. To this end, Rieffel introduced a first noncommutative analogue of the Gromov--Hausdorff distance between {\qcms s}. The search for an analogue of the Gromov--Hausdorff distance applicable to metric noncommutative geometry and possessing certain desirable properties, missing from the very important first analogue, eventually led the third author to the discovery of the \emph{Gromov--Hausdorff propinquity}.

The Gromov--Hausdorff propinquity is a complete metric on the class of {\qcms s}, up to full quantum isometries.
\begin{definition}[{\cite{Latremoliere13}}]
  Two {\qcms s} $(\A,\Lip_\A)$ and $(\B,\Lip_\B)$ are \emph{fully quantum isometric} if there exists a *-isomorphism $\pi : \A\rightarrow \B$ such that $\Lip_\B\circ\pi = \Lip_\A$.
\end{definition}

The class map which sends a compact metric space $(X,d)$ to the (classical) {\qcms} $(C(X),\Lip_d)$ is a homeomorphism from the Gromov--Hausdorff distance topology to the Gromov--Hausdorff propinquity topology. The Gromov--Hausdorff propinquity thus provides a functional analytic description of the Gromov--Hausdorff distance. An advantage of this presentation is that it opens up the possibility of defining the convergence of structures from other areas of noncommutative geometry, such as modules (analogues of vector bundles), group actions, and even spectral triples. The resulting metrics extend the Hausdorff distance to entire new classes of objects. Our focus in this paper is on the \emph{spectral propinquity}, a metric introduced by the third author between certain spectral triples called \emph{metric spectral triples}.

\begin{definition}[{\cite{Latremoliere18g}}]
  A \emph{metric} spectral triple $(\A,\Hilbert,D)$ is a spectral triple such that the metric $\Kantorovich{D}$, defined by Expression (\ref{distance-eq}), metrizes the weak* topology on the state space $\StateSpace(\A)$ of $\A$.
\end{definition}
We note that if $(\A,\Hilbert,D)$ is a metric spectral triple, then in particular, $\A$ must be represented faithfully on $\Hilbert$, namely for all $a\in\A$, if $\forall \xi \in \Hilbert \quad a\xi = 0$, then $a = 0$.

\bigskip

The spectral propinquity is a metric between metric spectral triples, up to unitary equivalence, which thus opens up the possibility of discussing approximations and perturbations of spectral triples, and in fact, is a metric on Riemannian structures and their noncommutative analogues. As we shall see in this work, the spectral propinquity induces a nontrivial topology on metric spectral triples.

\bigskip

Spectral triples over commutative C*-algebras also provide a new way to study the geometry of certain singular spaces. A prime example is given by the construction of spectral triples on certain fractals by the second author and his collaborators. In \cite{Lapidus94,Lapidus97}, Lapidus initiated a program to apply ideas from noncommutative geometry to the study of fractals. In \cite{Lapidus08}, The second author, together with {E}. Christensen and {C}. Ivan, construct a spectral triple on the {\SiepG} --- a basic example for analysis on fractals. This spectral triple allows one to recover the Hausdorff dimension and the Hausdorff measure of the {\SiepG}, and brings the tools of noncommutative geometry to bear on fractal geometry. In particular, it is shown in \cite{Lapidus08} that the restriction of the metric induced by Expression (\ref{distance-eq}) to the space of points (i.e., the pure states) of the {\SiepG} is the geodesic distance on the {\SiepG}. Similar results are obtained in \cite{Lapidus08} about a certain class of infinite trees and other quantum graphs. Other constructions and discussions of spectral triples on fractals can be found in \cite{sixfourfifth,tenhalf,fourhalf,twohalf,twoquarter}.

In \cite{Lapidus14}, the second author and J. Sarhad extended the construction of \cite{Lapidus08} to more general fractals, including the harmonic gasket introduced by Kusuoka and Kigami, in \cite{Kusuoka89,fifteenhalf} and \cite{elevenhalf}. In particular, the harmonic gasket is an example of a measurable Riemannian geometry, with an appropriate notion of volume (the Kusuoka measure), gradient, Laplacian, geodesic distance, and more, which makes it the fractal closest to a manifold, even though it is not a topological manifold, let alone a differentiable one. As we alluded to earlier, the spectral triples constructed in \cite{Lapidus08,Lapidus14} also recover, again via Expression (\ref{distance-eq}), the geodesic distance on the harmonic and {\SiepG s}, and are thus examples of metric spectral triples.

The authors of \cite{Lapidus14} also introduced a set of axioms under which their construction and results continue to hold. The resulting geometric objects are called \emph{piecewise $C^1$-fractal curves} in the present paper. Our own results will be established for this class of fractal--type ``manifolds''.

\begin{figure}\label{SiepG-fig}
  \centering
  \begin{tabular}{m{8cm} m{1cm} m{2cm}}
    \includegraphics[height=0.8in]{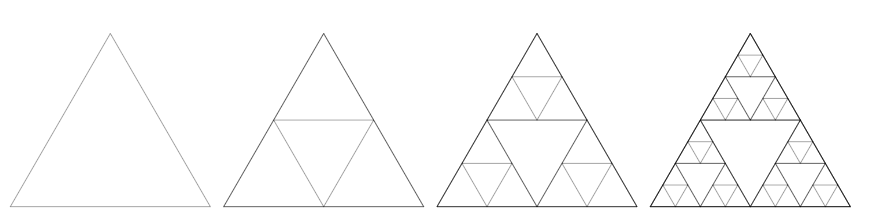} & \vspace*{1.6cm} $\ldots \longrightarrow$ \vspace*{0.3in} & \includegraphics[height=0.8in]{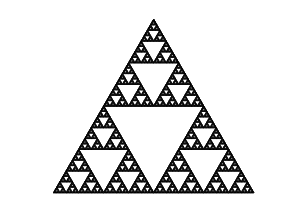}
  \end{tabular}
  \caption{The {\SiepG} is a limit of graphs in the plane for the Hausdorff distance}
\end{figure}

\bigskip

Now, fractals such as the {\SiepG} or the harmonic gasket are attractors of certain iterated function systems, and are limits of certain natural, and easier to understand, compact subsets of Euclidean spaces. Some approximations of these fractals for the Hausdorff distance are in fact basic tools in their study. For instance, the {\SiepG} is the limit, for the Hausdorff distance, of certain finite graphs in the plane; see Figure \ref{SiepG-fig}. A natural question arises:
\begin{quote}
  Is the spectral triple constructed on certain fractals which are limits of certain natural subsets of Euclidean spaces in \cite{Lapidus14} a limit, in some sense, of spectral triples on these natural approximation sets?
\end{quote}
This question is nontrivial since it first requires some notion of approximation for spectral triples. Yet, the development of the spectral propinquity from functional analytic and noncommutative metric geometry considerations provides us with the tool of the spectral propinquity to address this question. Therefore, the core concept which enables us to discuss the convergence of spectral triples is metric convergence.

\bigskip

In the process of precisely formulating our results, we were led naturally to introducing the notion of an \emph{approximation sequence} $(X_n)_{n\in\N}$ of a given piecewise $C^1$-fractal--curve $X$, relative to a parametrization of $X$; see Definition \ref{approx-sequence-def} below. Accordingly, our main results can informally be described as stating that given any approximating sequence $(X_n)_{n\in\N}$ of $X$, the sequence of corresponding spectral triples converges, in the sense of the spectral propinquity, to the spectral triple associated with $X$.

\bigskip

In closing this part of the introduction, we point out that our work and methods are naturally relevant to the study of analysis and probability theory on fractals, \cite{zero, ninehalf,elevenhalf,Kigami93,thirteenhalf,Kigami08,fourteenthreequarter,fourteenfourfifth,Kusuoka89,fifteenhalf,thirtyfourhalf,Lapidus94,Lapidus97, thirtyfiveoneeigth,thirtyfivehalf}, as well as to the physical and computational investigation of fractals and random media \cite{zero,fiftyninehalf,fourteenthreequarter,fourteenfourfifth,fifteenhalf,thirtyfourhalf,sixtytwohalf} and the references therein. Indeed, fractal models of natural objects and phenomena (such as clouds, coastlines, mountains, rivers, trees, plants, lungs, newtworks of blood vessels, lightning, galaxy clusters, etc) were developed with the understanding that, from a physical point of view, these natural objects may not truly be fractal, but can be well modeled by such a geometry up to a certain scale \cite{fiftyninehalf,sixtytwohalf}. It is therefore important to investigate how well these geometries are approximated by simpler finite graphs or piecewise smooth geometries.

\bigskip

Our work is organized in two main sections, which follow the natural path for the construction of the spectral propinquity in \cite{Latremoliere18g}. First, we discuss metric convergence for the class of fractals introduced in \cite{Lapidus14}. Following the original construction provided in \cite{Lapidus08}, we begin by giving detailed background on the {\SiepG}, as our first main example and also because it serves as a motivation and an important test case. We also discuss the fractal-like spaces from \cite{Lapidus14}, which we call \emph{piecewise $C^1$-fractal curves}. Another important example of such spaces, besides the {\SiepG}, is the harmonic gasket introduced by Kusuoka in \cite{Kusuoka89,fifteenhalf} and Kigami in \cite{elevenhalf}.

We then explain what we will mean by approximation stages for piecewise $C^1$-fractal curves, modeled after the natural approximations of the {\SiepG} by finite graphs, a fact which also plays a key role in the harmonic and spectral analysis of fractals, as well as in the study of diffusions (i.e., the analysis of a suitable analogue of Brownian motion) and probability theory on fractals; see \cite{zero,ninehalf,fifteenhalf,thirteenhalf,Kigami93,fourteenthreequarter,fourteenfourfifth,Kigami08,Kusuoka89,thirtyfourhalf,Lapidus97,thirtyfiveoneeigth,thirtyfivehalf}. Now, the metric we endow all these spaces with is the (intrinsic) geodesic distance, not the restriction of the Euclidean distance to them --- which is indeed different, in general. We establish the convergence of the approximation stages to piecewise $C^1$-fractal curves for the Gromov--Hausdorff distance. As we discussed, this immediately implies a result about convergence for the propinquity. However, in order to apply our framework to the convergence of spectral triples, we need to provide a functional analytic proof, which concludes this first main section (Section 2). This first section also includes the basic definition of the Gromov--Hausdorff propinquity.

\bigskip

Second, we prove that the spectral triples of \cite{Lapidus14} associated with piecewise $C^1$-fractal curves are limits, with respect to the spectral propinquity, of certain spectral triples constructed on the approximation stages of these fractals. This implies, in particular, that the spectral triples on the {\SiepG} \cite{Lapidus08} and on the harmonic gasket \cite{Lapidus14} are limits of spectral triples on their naturally approximating graphs. In order to obtain this result, we first recall the construction from \cite{Lapidus14}, and introduce the spectral triples on the approximation stages. We then follow the process from \cite{Latremoliere18g}. A metric spectral triple consists of
\begin{itemize}
\item a {\qcms},
\item a Hilbert space with an extra norm defined on some dense subspace, given by the the graph norm of the Dirac operator,
\item a *-representation of the {\qcms} on the Hilbert space,
\item a group action of $\R$ on the Hilbert space obtained by exponentiating $i$ times the Dirac operator.
\end{itemize}
Each of these ingredients requires a form of convergence, each devised by the third author, for different structures from noncommutative metric geometry. Of course, the convergence of {\qcms s} is defined by the propinquity. The convergence of the Hilbert spaces with the graph norms of the Dirac operators of spectral triples is given by the \emph{modular propinquity} \cite{Latremoliere16c}. The convergence of the *-representations in spectral triples is defined by the \emph{metrical propinquity} \cite{Latremoliere18g}. Including an appropriately defined notion of covariant metric propinquity \cite{Latremoliere18b}, called the spectral propinquity \cite{Latremoliere18g}, we deduce the convergence of the spectral triples.

We close this introduction by providing notation which will be used throughout this paper.

\bigskip

\begin{notation}
  The set of natural numbers, denoted by $\N$, starts with zero: $\N = \{ 0,1,2,\ldots\}$. We write $\Nbar = \N \cup \{\infty\}$, where $\infty\not\in \N$, with the order relation $\leq$ such that $\infty$ is the greatest element of $\Nbar$ and the restriction of $\leq$ to $\N$ is the usual order of $\N$.

  Moreover, for any $n\in \Nbar$, we set $\Nbar_n = \{ k \in \Nbar : k \leq n \}$ and $\N_n = \{ k \in \N : k \leq n \}$.
\end{notation}

\begin{notation}
  The norm of a normed vector space $E$ is denoted by $\norm{\cdot}{E}$; if $E$ is an inner product space, then $\norm{\cdot}{E}$ is the norm induced by this inner product. If $E = C(X)$ is the C*-algebra of all $\C$-valued continuous functions over a compact Hausdorff space $X$, then $\norm{\cdot}{C(X)}$ is the supremum norm over $X$.

  For any unital C*-algebra $\A$, we denote by $\StateSpace(\A)$ the state space of $\A$, i.e. the set of all positive linear functionals of norm $1$ over $\A$. In particular, by the Riesz--Radon representation theorem, $\StateSpace(C(X))$, for $X$ a compact Hausdorff space, is naturally identified with the space of Radon probability measures over $X$.

  The space of self-adjoint element $\{ a \in \A : a^\ast = a \}$ of $\A$ is denoted by $\sa{\A}$; if $\A = C(X)$ is the space of $\C$-valued continuous functions on $X$, as above, then $\sa{\A}$ is the algebra of $\R$-valued continuous functions on $X$.
\end{notation}

\begin{notation}
  The norm of a continuous linear endomorphism $T$ on a normed vector space $E$ is denoted by $\opnorm{T}{}{E}$ --- so that we need not always refer to the Banach algebra of such operators explicitly.
\end{notation}

\begin{notation}\label{Hausdorff-notation}
  If $(E,d)$ is a metric space, then $\Haus{d}$ denotes the Hausdorff distance  \cite{Hausdorff} on the space of all nonempty closed subsets of $(E,d)$; see, e.g., \cite{sixquarterepsilon, Gromov81, Gromov, thirtytwoquarter}. If the space $E$ is a vector space endowed with a norm $\norm{\cdot}{E}$, then $\Haus{\norm{\cdot}{E}}$ is the Hausdorff distance for the metric induced on $E$ by $\norm{\cdot}{E}$. 
\end{notation}
\section{Graph Approximations of Piecewise $C^1$-Fractal Curves for the Gromov--Hausdorff Distance}

\subsection{The {\Siep} Gasket}

\begin{figure}
  \centering
  \includegraphics[height=2in]{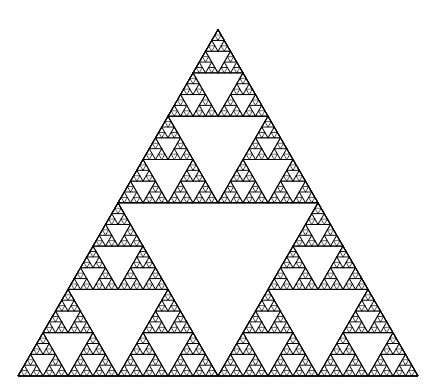}
  \caption{The {\SiepG}}
\end{figure}

The {\SiepG} $\SG{\infty}$ is a fractal, constructed as the attractor set of an iterated function system (IFS) of affine functions of the plane. Specifically, let
\begin{equation*}
  v_0 = \begin{pmatrix} 0 \\ 0 \end{pmatrix}\text{, }v_1 = \begin{pmatrix} 1 \\ 0 \end{pmatrix} \text{ and }v_2 = \begin{pmatrix} \frac{1}{2} \\ \frac{\sqrt{3}}{2} \end{pmatrix} \text{.}
\end{equation*}
We write $V_0 = \{ v_0, v_1, v_2 \}$. Let $\Delta_{0,1} = \SG{0}$ be the boundary of the convex hull of $V_0$ in $\R^2$ --- i.e., $\Delta_{0,1}$ is an equilateral triangle in the plane, whose edges have length $1$. Let $L_0 = \{ \Delta_{0,1} \}$.

We define three similitudes of the plane by letting for each $j\in\{0,1,2\}$,
\begin{align*}
  T_j &: x \in \R^2 \longmapsto \frac{1}{2} \left( x + v_j \right) \in \R^2 \text{.} \\
\end{align*}

We will use an explicit construction of $\SG{\infty}$ as a limit of finite graphs in $\R^2$, defined inductively. For all $n\in\N$, $n>0$, we set
\begin{equation*}
  L_n = \left\{ \Delta_{n,j} : j \in \{1,\ldots,3^n\} \right\} \text{,}
\end{equation*}
where
\begin{itemize}
\item $\Delta_{n+1,j + r 3^n} = T_r\Delta_{n,j}$, for all $j\in\{1,\ldots,3^{n}\}$ and $r\in\{0,1,2\}$,
\item $V_{n+1} = \bigcup_{r=0}^2 T_r V_n$.
\end{itemize}

For each $n\in\N$, we define the set
\begin{equation*}
  \SG{n} = \bigcup L_n = \bigcup_{j=1}^{3^n}\Delta_{n,j} \text{.}
\end{equation*}

We observe for later use that, by induction, for all $n\in\N$:
\begin{enumerate}
\item $\Delta_{n,j}$ is an equilateral triangle whose edges have length $\frac{1}{2^{n}}$;
\item the set of all vertices of the triangles in $L_n$ is $V_n$;
\item if $j \in \{1,\ldots,3^n\}$, then
  \begin{equation*}
    \Delta_{n,j}\subseteq \bigcup_{r=1}^3 \Delta_{n+1, 3(j-1) + r}\subseteq \co{\Delta_{n,j}} \text{;}
  \end{equation*}
\item $\SG{n} \subseteq \SG{n+1}$;
\item if $j,k \in \{1,\ldots,3^n\}$, and $j\not=k$, then $\Delta_{n,j}\cap\Delta_{n,k}$ is empty or a singleton containing the common vertex to both triangles;
\item $\SG{n}$ is path connected.
\end{enumerate}

All of these observations above follow from the fact that affine bijections preserve triangles, scale length (here, by $\frac{1}{2}$), and preserve intersections; they all can be proved by induction.

\medskip

Our construction implies the following key metric property:
\begin{equation*}
  \forall n,m \in \N \quad m\geq n \implies \Haus{\norm{\cdot}{\R^2}}\left(\SG{m},V_n\right) \leq \frac{1}{2^{n}}\text.
\end{equation*}

\medskip

We now define the \emph{\SiepG}, using the notation of this section.
\begin{definition}
  The {\SiepG} $\SG{\infty}$ is the closure of $\bigcup_{n\in\N}\SG{n}$.
\end{definition}

We further define $V_\infty = \bigcup_{n\in\N}V_n$.

\medskip

By construction, the {\SiepG} is compact. It can also easily be checked that $\SG{\infty}$ is invariant under the map
\begin{equation*}
  X \subseteq\R^2 \mapsto \bigcup_{r=0}^2 T_r X \text{;}
\end{equation*}
so it is indeed the attractor of the iterated functions system $(T_0,T_1,T_2)$ and is, in fact, a self-similar set (see, e.g., \cite{sixquarter}); namely,
\begin{equation*}
  \SG{\infty} = \bigcup_{r=0}^2 T_r \SG{\infty} \text{.}
\end{equation*}

\medskip

It is immediate from our construction that
\begin{equation*}
  \forall n \in \N, \quad \Haus{\norm{\cdot}{\R^2}}(\SG{\infty},V_n) \leq \frac{1}{2^{n}}\text,
\end{equation*}
where we have used Notation \ref{Hausdorff-notation}. The set $V_\infty$ is therefore dense in $\SG{\infty}$. We also note that, by construction,
\begin{equation}\label{sg-cv-eq}
  \forall n\in\N, \quad \Haus{\norm{\cdot}{\R^2}}\left(\SG{n},\SG{\infty}\right) \leq \frac{1}{2^n} \text{;}
\end{equation}
hence, $\SG{\infty}$ is the limit of $(\SG{n})_{n\in\N}$ for $\Haus{\norm{\cdot}{\R^2}}$.

\bigskip

However, for each $n\in\Nbar$, we will work with the intrinsic metric on the set $\SG{n}$, which is different from the restriction to $\SG{n}$ of the metric of $\R^2$. Indeed, since for each $n\in\Nbar$, the set $\SG{n}$ is path connected, we can define for any two $x,y \in \SG{n}$,
\begin{multline*}
  d_n(x,y) = \\
  \inf\left\{ \mathrm{length}(\gamma) : \gamma : [0,1]\rightarrow \SG{n}, \gamma(0) = x, \gamma(1) = y, \text{$\gamma$ continuous} \right\}\text.
\end{multline*}
Here and henceforth (see, e.g., \cite{Gromov,sixtytwothreequarter}), for any natural number $p\geq 1$, and for any curve in a compact subset $X$ of $\R^p$, i.e., a  continuous map $\gamma : [0,1]\rightarrow X$, we define the \emph{length} of $\gamma$ by
\begin{multline*}
  \mathrm{length}(\gamma) = \\
  \sup\left\{ \sum_{j=0}^k \norm{\gamma(t_j) - \gamma(t_{j+1})}{\R^p} : k\in\N, 0=t_0 < t_1 < \ldots < t_k = 1\right\}\text{,}
\end{multline*}
allowing for the value $\infty$ for $\mathrm{length}(\gamma)$ (the curves with finite length are called \emph{rectifiable}), and where $\norm{\cdot}{\R^p}$ is the Euclidean norm on $\R^p$.

By the Hopf--Rinow theorem for length spaces \cite{Gromov,sixtytwothreequarter}, and since $\SG{n}$ is path connected and compact, there exists a continuous function $\gamma : [0,1]\rightarrow \SG{n}$ which is a geodesic from $\gamma(0) = x$ to $\gamma(1) = y$; i.e., for all $t\leq t' \in [0,1]$, we have $\mathrm{length}(\gamma|[t,t']) = d_n(\gamma(t),\gamma(t')) = \lambda|t-t'|$, where $\lambda=\mathrm{length}(\gamma)$. This last equality shows that $\gamma$ is injective. We will use these observations in several proofs below.

\bigskip

In general, the canonical inclusion of $\SG{n}$ into $\SG{\infty}$ is not an isometry from $(\SG{n},d_n)$ to $(\SG{\infty},d_\infty)$. For instance, we see that
\begin{equation*}
  d_0\left(v_2,\begin{pmatrix} \frac{1}{2} \\ 0 \end{pmatrix} \right) = \frac{3}{2}\text{, yet }
  d_\infty \left(v_2,\begin{pmatrix} \frac{1}{2} \\ 0 \end{pmatrix} \right)= d_1\left(v_2,\begin{pmatrix} \frac{1}{2} \\ 0 \end{pmatrix} \right) = 1 \text{.}
\end{equation*}
This simple computation also shows that, of course, for any $n\in\Nbar$, the space $(\SG{n},d_n)$ is not a metric subspace of $\R^2$ (with its usual metric); namely, while $\SG{n}$ is a subset of $\R^2$, the restriction to $\SG{n}$ of the usual Euclidean metric on $\R^2$ is not equal to $d_n$. It is obvious, nonetheless, that
\begin{equation*}
  \forall n\in\Nbar, \quad \forall x,y \in \SG{n}, \quad \norm{x-y}{\R^2} \leq d_\infty(x,y) \leq d_n(x,y) \text{.}
\end{equation*}

However, we make the following observation, which will prove helpful later in this work.

\begin{lemma}\label{vertices-lemma}
  For all $n\in\N$, the metrics $d_n$ and $d_\infty$ agree on $V_n$. In fact, any geodesic between two elements of $V_n$ in $(\SG{\infty},d_\infty)$ is also a geodesic in $(\SG{n},d_n)$.
\end{lemma}

\begin{proof}
  As seen above, for any fix $n\in\N$, and for all $j,k \in \{1,\ldots,3^n\}$ with $j\not=k$, the triangles $\Delta_{n,j}$ and $\Delta_{n,k}$ intersect at most at one point, which is then a common vertex of both triangles. Thus, a continuous curve from a point in $\Delta_{n,j}$ to a point $\Delta_{n,k}$ must pass by a vertex of each of these two triangles, since $\SG{n} = \bigcup_{j=1}^{3^n}\Delta_{n,j}$.
  
  Let $v,w \in V_n$ be two distinct vertices of $\SG{n}$. If there exists $j\in \{1,\ldots,3^n\}$ such that $v,w \in \Delta_{n,j}$, then it is immediate that $d_\infty(v,w) = d_n(v,w) = 2^{-n}$ because $v$ and $w$ are joined in $\SG{\infty}$ and $\SG{n}$ by an edge of a triangle in $L_n$, and as a result of the triangle inequality.

  Now, suppose $v\in\Delta_{n,j}$ and $w \in \Delta_{n,k}$, for some $j,k \in \{1,\ldots, 3^n\}$ with $j\not=k$. Let $\gamma$ be a rectifiable curve in $\SG{\infty}$ from $v$ to $w$ which is not contained in $\SG{n}$; so there exists $c \in V_\infty\setminus \SG{n}$ also within the range of $\gamma$. Now, $c$ lies in the convex hull of $\Delta_{n,r}$, for some $r \in \{1,\ldots,3^n\}$. Thus, $\gamma$ had to pass by two vertices $a,b$ of $\Delta_{n,r}$: if $r=j$ or $k$, then one of these vertices is $v$ or $w$, and somehow we must exit $\Delta_{n,r}$, which can only occur via a vertex. Otherwise, we entered and exited $\Delta_{n,r}$, which can only occur via a vertex. It follows that the part of $\gamma$ between $a$ and $b$ is then longer than the straight line $[a,b]$ and, hence, $\gamma$ is not a geodesic from $v$ to $w$.

  This completes the proof of the lemma, by contraposition.
\end{proof}


Without any obvious canonical space to isometrically embed $(\SG{n},d_n)$ and $(\SG{\infty},d_\infty)$ into, we will discuss the notion of convergence in the sense of the Gromov--Hausdorff distance $\GH$. Our key assumption is that, for every $n\in\N$, the spaces $(V_n,d_n)$ and $(V_n,d_\infty)$ are actually the same metric space, by Lemma \ref{vertices-lemma}. It is then easy to check that (with the use of Notation \ref{Hausdorff-notation})
\begin{equation}\label{v_n-to-sgn-eq}
  \Haus{d_n}(\SG{n},V_n) \leq 2^{-n}\text{ and }\Haus{d_\infty}(\SG{\infty},V_n) \leq 2^{-n}\text{;}
\end{equation}
so that
\begin{align*}
  \GH((\SG{n},d_n),(\SG{\infty},d_\infty))
  &\leq \Haus{d_n}(\SG{n},V_n) + \GH((V_n,d_n),(V_n,d_\infty)) \\
  &\quad + \Haus{d_\infty}(V_\infty,\SG{\infty}) \\
  &\leq 2^{-n} + 0 + 2^{-n} = 2^{-n+1} \text{,}
\end{align*}
and hence,
\begin{equation}
  \lim_{n\rightarrow\infty} \GH((\SG{n},d_n),(\SG{\infty},d_\infty)) = 0\text.
\end{equation}

The picture painted in this section can be fruitfully generalized to include such interesting examples as the harmonic gasket, as we shall see in the next subsections.

\subsection{Piecewise $C^1$-Fractal Curves}

The {\SiepG} is an example of a class of compact length subspaces introduced in \cite{Lapidus14} as examples of spaces for which a natural spectral triple may be constructed.

For the following definition, the endpoints of a curve $\gamma$ are $\gamma(0)$ and $\gamma(1)$. A curve $\gamma$ is said to be a \emph{concatenation} of a sequence $(\gamma_j)_{j\in\N}$ of curves when there exists a strictly increasing sequence $(t_j)_{j\in\N}$ in $[0,1]$, with $t_0 = 0$ and $\lim_{j\rightarrow\infty}t_j=1$, such that, for each $j\in\N$, the curve $\gamma$ restricted to $[t_j,t_{j+1}]$ is the map $t \in [0,1] \mapsto \gamma_j\left( \frac{ t - t_j }{ t_{j+1} - t_j }  \right)$. The concatenation of finitely many curves is defined similarly, by using a finite subdivision $t_0 = 0 < t_1 < \ldots < t_k = 1$ of $[0,1]$.

\begin{definition}\label{fractal-curve-def}
  A \emph{piecewise $C^1$-fractal curve} $X$ is a compact path connected subset of $\R^n$ such that, for some sequence $(C_j)_{j\in\N}$ of rectifiable $C^1$-curves in $\R^n$ with $\lim_{j\rightarrow\infty} \mathrm{length}(C_j) = 0$, the following assertions hold:
  \begin{enumerate}
  \item $X = \text{the closure of } \bigcup_{j \in \N} \range{C_j}$,
  \item there exists a dense subset $\mathcal{B}$ of $X$ (for the topology induced by the geodesic distance on $X$) consisting of endpoints of the curves in the sequence $(C_j)_{j\in\N}$ and such that for all $p \in \mathcal{B}$ and $q\in X$, one of the geodesics from $p$ to $q$ in $X$ is a curve obtained as a concatenation of a possibly finite subsequence $(C_j)_{j \in \N}$.
  \end{enumerate}

  \noindent The sequence $(C_j)_{j\in\N}$ is called a \emph{parametrization} of $X$.
\end{definition}

The {\SiepG} is a piecewise $C^1$-fractal curve given by a parametrization using the edges of the triangles $\Delta_{n,r}$, for all $n\in\N$ and $r \in \{1,\ldots,3^n\}$.

\begin{theorem}[{\cite[Proposition 2]{Lapidus14}}]\label{SiepG-param-thm}
  The {\SiepG} $\SG{\infty}$ is a piecewise $C^1$-fractal curve, with parametrization $(R_j)_{j\in\N}$ given, for each $j \in \N$, by either of the two affine functions from $[0,1]$ onto 
  \begin{itemize}
    \item the bottom edge of $\Delta_{n,j}$, if $j = \varkappa(n,r)$ for some $(n,r) \in \Xi$,
    \item the right edge of $\Delta_{n,j}$, if $j = \varkappa(n,r) + 1$ for some $(n,r)\in\Xi$,
    \item the left edge of $\Delta_{n,j}$, if $j = \varkappa(n,r) + 2$ for some $(n,r) \in \Xi$,
  \end{itemize}
  with
  \begin{equation*}
    \Xi := \left\{ (n,r) \in \N^2 : n\in \N, r \in \{0,\ldots,3^n-1\} \right\}
  \end{equation*}
  and $\varkappa(n,r) := 3\left(\sum_{k=0}^{n-1} 3^k + r\right)$, for all $(n,r) \in \Xi$; by convention $\chi(0,0) = 0$.
\end{theorem}

\begin{convention}
  We will refer to the parametrization of the {\SiepG} $\SG{\infty}$ in Theorem \ref{SiepG-param-thm} as the \emph{standard parametrization} of $\SG{\infty}$.
\end{convention}

Besides the {\SiepG}, another very important example of a piecewise $C^1$-fractal curve is given by the \emph{harmonic gasket}, which is a fractal with what can be called a ``measurable Riemannian geometry'', an idea initiated and developed by Kusuoka \cite{Kusuoka89} and later, by Kigami \cite{Kigami08}; see also \cite{Lapidus94,Lapidus97,ninehalf}.

The natural Dirichlet form over the {\SiepG} is the pointwise limit, on some dense subspace $\mathcal{D}$ of $C(\SG{\infty})$, of the sequence of quadratic forms $\left(\mathcal{E}_n\right)_{n\in\N}$, defined, for all $n\in\N$, and for all functions $u \in C(\SG{\infty})$, by
\begin{equation*}
  \mathcal{E}_n(v) = \left(\frac{5}{3} \right)^n \sum_{\substack{p,q \in V_n \\ p \underset{n}{\sim} q}} \left( u(p)-u(q) \right)^2 \text.
\end{equation*}
Thus, there exists a operator $\Delta$, defined on $\mathcal{D}$ and valued in $C(\SG{\infty}$, such that
\begin{equation*}
  \forall u \in \mathcal{D} \quad \mathcal{E}(u) = \inner{\Delta u}{u}{} \text.
\end{equation*}
As $\mathcal{E}$ is indeed the Dirichlet form associated with the analogue of the Brownian motion on the {\SiepG} (see \cite{zero,Kusuoka89,fifteenhalf} and refrences therein), the operator $\Delta$ can be seen as an analogue of the Laplacian for the {\SiepG}. Kigami then proved in \cite{Kigami93,thirteenhalf} that, given any function $f : V_0 \rightarrow\R$, there exists a unique function $u \in C(\SG{\infty}$ such that $\Delta u = 0$ and $u|0 = f$ --- the analogue of the Poisson problem for the {\SiepG} always has a unique solution.

In particular, there exists unique harmonic functions $u_1,u_2,u_3$ on $\SG{\infty}$ such that
\begin{equation*}
  u_j(v_k)
  =
  \begin{cases}
    1 \text{, if $j = k$,}\\
    0 \text{, otherwise.}
  \end{cases}
\end{equation*}

We note that $u_1 + u_2 + u_3 = 1$ on $\SG{\infty}$, by the uniqueness of the solution to the Poisson problem on $\SG{\infty}$; hence, $\{ u_1,u_2,u_3\}$ is a partition of unity in $C(\SG{\infty})$. The function
\begin{equation*}
  \Phi : x \in \SG{\infty} \longmapsto \frac{\sqrt{2}}{2} \left( \begin{pmatrix} u_1(x) \\ u_2(x) \\ u_3(x) \end{pmatrix} - \begin{pmatrix} 1 \\ 1 \\ 1 \end{pmatrix} \right)
\end{equation*}
is a continuous injection on $\SG{\infty}$, and thus an homeomorphism from the {\SiepG} $\SG{\infty}$ onto its image in the affine subspace of $\R^3$,
\begin{equation*}
  \left\{ (x,y,z) \in \R^3 : x+y+z = 1 \right\}\text,
\end{equation*}
viewed as a subset of $\R^2$. By definition, the \emph{harmonic gasket} $\HG{\infty}$ is the image of $\SG{\infty}$ by $\Phi$ (see \cite{elevenhalf}):
\begin{equation*}
  \HG{\infty} = \Phi\left( \SG{\infty} \right) \text{.}
\end{equation*}

\begin{figure}
  \centering
  \includegraphics[height=2in]{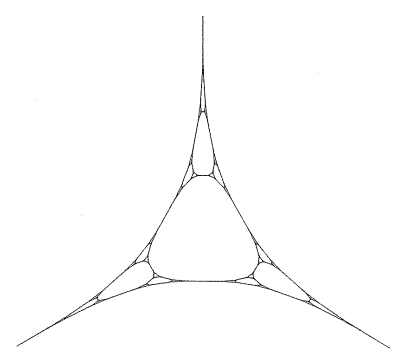}
  \caption{The harmonic gasket}
\end{figure}

We recall that, by \cite{Lapidus14}, we have the following result.
\begin{theorem}[{\cite[Proposition 3]{Lapidus14}}]\label{harmonic-gasket-c1-thm}
  The harmonic gasket $\HG{\infty}$ is a piecewise $C^1$-fractal curve, with parametrization $\left(\Phi\circ R_j\right)_{j\in\N}$, where $(R_j)_{j \in \N}$ is the standard parametrization of the {\SiepG}  $\SG{\infty}$.
\end{theorem}

\begin{convention}
  The sequence $(\Phi\circ R_j)_{j \in \N}$ is called the \emph{standard parametrization} of $\HG{\infty}$.
\end{convention}

Theorem \ref{harmonic-gasket-c1-thm} makes use, in particular, of the existence of a continuously differentiable (i.e., $C^1$) geodesic between any two points $x,y$ in $\SG{\infty}$. Such a geodesic is usually not unique, and is not $C^2$. In fact, there are typically infinitely many such geodesics.

\bigskip

In this work, we are concerned with certain approximations of piecewise $C^1$-fractal curves by finite unions of their constituent rectifiable curves, in a manner which includes and generalizes the spaces $\SG{n}$ approximating $\SG{\infty}$.

\begin{definition}\label{approx-sequence-def}
  Let $X$ be a piecewise $C^1$-fractal curve with parametrization $(C_j)_{j\in\N}$, as in Definition \ref{fractal-curve-def}. Then, an \emph{approximation sequence} of $X$ \emph{compatible with $(C_j)_{j\in\N}$} is a strictly increasing function $B : \N \rightarrow \N$ such that, for every $\varepsilon > 0$, there exists $N\in\N$ such that if $n\geq N$, and letting
  \begin{itemize}
  \item $X_n = \bigcup_{j = 1}^{B_n} \range{C_j}$,
  \item $V_n$ be the set of all endpoints $\{ C_j(0),C_j(1) : j = 1,\ldots,B_n \}$ of the curves $R_1$,\ldots,$R_{B_n}$,
  \item $d_n$ be the geodesic distance on $X_n$ (in particular, $X_n$ is path connected),
  \end{itemize}
  the following properties hold:
  \begin{enumerate}
  \item the restriction of $d_\infty$ to $V_n\times V_n$ is $d_n$,
  \item $\Haus{d_n}(V_n,X_n) < \varepsilon$.
  \end{enumerate}
\end{definition}
Of course, geometrically, and using the notation of Definition \ref{approx-sequence-def}, we mean to approximate the piecewise $C^1$-fractal curve $X$ by the sequence of subsets $(X_n)_{n\in\N}$ of $X$. But it will be very helpful to keep track of the constituent curves.

Our prototype for Definition \ref{approx-sequence-def} is given by the following theorem.

\begin{theorem}\label{SiepG-approx-seq-thm}
  The sequence $\left(\frac{3}{2}\left(3^{n+1}-1\right)\right)$ is an approximation sequence for the {\SiepG} $\SG{\infty}$ compatible with its standard parametrization.
\end{theorem}

\begin{proof}
  Let $B_n = 3\frac{3^{n+1}-1}{2} = 3\sum_{j=1}^{n} 3^j$. For each $n\in\N$, the set $\bigcup_{j=1}^{B_n} R_j$ is the set $\SG{n}$. Thus, Condition (1) of Definition \ref{approx-sequence-def} is satisfied, in light of Equation (\ref{v_n-to-sgn-eq}), and Condition 2 follows from Lemma \ref{vertices-lemma}. 
\end{proof}

We can then use the approximation sequence for the standard parametrization of $\SG{\infty}$ in order to obtain an approximation sequence of the harmonic gasket. The approximation of $\HG{\infty}$ will be obtained by means of the following sets:
\begin{convention}
  For each $n \in \N$, we set $\HG{n} = \Phi(\SG{n})$.
\end{convention}

\begin{theorem}\label{HarmG-approx-seq-thm}
  The sequence $\left(\frac{3}{2}\left( 3^{n+1} - 1 \right)\right)_{n\in\N}$ is an approximation sequence for $\HG{\infty}$ adapted to the standard parametrization of $\HG{\infty}$.
\end{theorem}

\begin{proof}
  Let $\varepsilon > 0$. There exists $N\in\N$ such that for all $n\geq N$, the length of $\Psi\circ R_n$ is strictly less than $\varepsilon$. Let $n \geq \max\{ N, \frac{3}{2}\left(3^{n+1}-1\right) \}$.
  
  By construction, $\SG{n} = \bigcup_{j=1}^{3^n} \Delta_{n,j}$ and thus, $\HG{n} = \bigcup_{j=1}^{3^n} \Phi(\Delta_{n,j})$. By definition, if $x \in \HG{n}$, then $x\in \Phi\left(\Delta_{n,j}\right)$ for some $j \in \{1,\ldots,3^n\}$. Hence, for some $j\geq N$, the point $x$ lies on $\mathrm{ran}\left(\Phi \circ R_j \right)$, whose length is strictly less than $\varepsilon > 0$. It follows that $d_n(x,v) < \varepsilon$, for any endpoint $v$ of $R_j$, i.e., for any element of $V_n$. This proves that $\Haus{d_n}(V_n,\HG{n}) < \varepsilon$.

  \bigskip  

  Now, let $x,y \in V_n$. If $x,y \in \Phi(\Delta_{n,j})$, for some $j\in\{1,\ldots,3^n\}$, then, letting $E$ be the edge of $\Delta_{n,j}$ from $\Phi^{-1}(x)$ to $\Phi^{-1}(y)$ (parametrized as an affine map), we have by \cite[Proposition 2]{Lapidus14} that $\Phi\circ E$ is a geodesic from $x$ to $y$ in $\SG{\infty}$. It is then, of course, a geodesic in $(\HG{n},d_n)$ as well; so that $d_n(x,y) = \mathrm{length}(\Phi\circ E) = d_\infty(x,y)$.
  
  Otherwise, let $\gamma$ be some injective rectifiable curve from $v$ to $w$. Let $\alpha = \Phi^{-1}\circ\gamma$, which is an injective rectifiable curve in $\SG{\infty}$ between two vertices in $\SG{n}$. As discussed in Lemma \ref{vertices-lemma} and owing to our work on the {\SiepG}, there exists a subdivision $0 = t_0 < t_1 < \ldots < t_k < t_{k+1} = 1$ of $[0,1]$ such that
  \begin{equation*}
    \forall t \in [0,1], \quad \alpha(t) \text{ is a vertex in $\SG{n}$ }  \iff \exists j \in \{ 0,\ldots,k+1 \}, \quad t = t_j \text{,}
  \end{equation*}
  since there are finitely many vertices in $\SG{n}$, and since $\alpha$ is injective, it does not cross the same vertex twice.
  We now define a new curve in $\SG{\infty}$. For each $j \in \{0, \ldots, k\}$, let $E_j$ be the affine parametrization of the edge from $\alpha(t_j)$ to $\alpha(t_{j+1})$ in the triangle of $L_n$ containing both of these vertices. We then set
  \begin{equation*}
    \beta : t \in [0,1] \longmapsto E_j\left( \frac{t-t_j}{t_{j+1}-t_j} \right), \text{ if $t\in[t_j,t_{j+1}]$.}
  \end{equation*}
  Let $\gamma' = \Phi\circ\beta$. The curve $\gamma'$ is rectifiable, from $x$ to $y$, contained in $\HG{n}$ by construction, and the length of $\beta$ between $\beta(t_j) = \gamma(t_j)$ and $\beta(t_{j+1}) = \gamma(t_{j+1})$ is the smallest among all curves between these two vertices, again by \cite[Proposition 2]{Lapidus14}; so that
  \begin{equation*}
    \mathrm{length}(\gamma') \leq \mathrm{length}(\gamma) \text{.}
  \end{equation*}

  Applying the above construction to a geodesic $\gamma$ from $x$ to $y$ in $\HG{\infty}$, we obtain a geodesic of $\HG{n}$ between $x$ and $y$, with the same length. Thus $d_n(x,y) = d_\infty(x,y)$, as desired. This completes the proof of the theorem.
\end{proof}

\bigskip

The pattern observed in Theorems \ref{SiepG-approx-seq-thm} and \ref{HarmG-approx-seq-thm} is abstracted in the following hypothesis, which will be assumed in the remainder of our work.

\begin{hypothesis}\label{main-hyp}
  Let $\FC{\infty}$ be a piecewise $C^1$-fractal curve with parametrization $(C_j)_{j\in\N}$; so that, in particular,
  \begin{equation*}
    \FC{\infty} = \text{closure of }\bigcup_{j\in\N} \range{C_j} \text{.} 
  \end{equation*}
  We denote its geodesic distance by $d_\infty$. We also denote the set of all the enpoints of the curves $C_j$ ($j\in\N$) --- which we call the vertices of $\FC{\infty}$ --- by $V_\infty$.

  Let $(B_n)_{n\in\N}$ be an approximation sequence for $\FC{\infty}$ adapted to the parametrization $(C_j)_{j\in\N}$, and set $B_\infty = \infty$. For each $n\in\N$, we write
  \begin{equation*}
    \FC{n} = \bigcup_{j = 0}^{B_n} \range{C_j} \text{.}
  \end{equation*}
  We also denote the geodesic distance on $\FC{n}$ by $d_n$. Last, we let
  \begin{equation*}
    V_n = \left\{ C_j(0), C_j(1) : j \in \{0,\ldots,B_n\} \right\}\text,
  \end{equation*}
  which we call the \emph{set of vertices} of $\FC{n}$, and we set $V_\infty = \bigcup_{j \in \N} V_j$ --- whose elements we refer to as the \emph{vertices} of $\FC{\infty}$.

  For every $n\in\Nbar$, we denote the Lipschitz seminorm on $\FC{n}$ induced by the geodesic distance $d_n$ by $\Lip_n$.

  Finally, for every $j \in \N$, we also denote the length of $C_j$ by $\lambda_j$.
\end{hypothesis}

\subsection{Metric Convergence}

Hypothesis \ref{main-hyp} implies that the sequence $(\FC{n},d_n)_{n\in\N}$ converges to $(\FC{\infty},d_\infty)$ for the Gromov--Hausdorff distance. By Definition \ref{fractal-curve-def}, if we let $\varepsilon > 0$, there exists $N\in\N$ such that, if $n \geq N$, then the following properties hold:
\begin{enumerate}
\item $(V_n,d_n)$ and $(V_n,d_\infty)$ are the same metric spaces,
\item $\Haus{d_n}(\FC{n},V_n) < \frac{\varepsilon}{2}$.
\end{enumerate}

Moreover, by Definition \ref{approx-sequence-def}, the set $V_\infty$ is dense in $(\FC{\infty},d_\infty)$, and since $V_\infty$ is the increasing union of the sequence $(V_n)_{n\in\N}$, there exists $N' \in \N$, such that if $n\geq N'$, then $\Haus{d_\infty}(\FC{\infty},V_n) < \frac{\varepsilon}{2}$. Thus, if $n\geq \max\{N,N'\}$, then we have the following properties:
\begin{itemize}
\item $\Haus{d_\infty}(\FC{\infty},V_n) < \frac{\varepsilon}{2}$, since $V_{N'}\subseteq V_n$, by construction,
\item $\Haus{d_n}(\FC{n},V_n) \leq \frac{\varepsilon}{2}$.
\end{itemize}

Therefore,
\begin{align*}
  \GH((\FC{n},d_n),(\FC{\infty},d_\infty))
  &\leq \Haus{d_n}(\FC{n},V_n) + \GH((V_n,d_n),(V_n,d_\infty)) \\
  &\quad + \Haus{d_\infty}(V_n,\FC{\infty}) \\
  &\leq \frac{\varepsilon}{2} + 0 + \frac{\varepsilon}{2} = \varepsilon \text{;}
\end{align*}
so that
\begin{equation*}
  \lim_{n\rightarrow\infty} \GH((\FC{n},d_n),(\FC{\infty},d_\infty)) = 0\text.
\end{equation*}

\bigskip

In order to apply our techniques to the convergence of spectral triples, it is necessary to bring our previous observation about metric convergence within the functional formalism of the propinquity --- as the spectral propinquity is built on this formalism. We begin by recalling the basic construction of the Gromov--Hausdorff propinquity; we refer to \cite{Latremoliere13,Latremoliere13b,Latremoliere14} for the details.

\bigskip

In order to motivate our construction, we begin by recalling the construction of the Gromov--Hausdorff distance \cite{Gromov81,Gromov}. The Gromov--Hausdorff distance between two compact metric spaces $(X,d_X)$ and $(Y,d_Y)$ is the infimum, over all choices of a compact metric space $(Z,d_Z)$ which contain an isometric copy of $(X,d_X)$ and $(Y,d_Y)$, of the Hausdorff distance, $\Haus{d_Z}$ \cite{Hausdorff}, between $X$ and $Y$ in $Z$. The first step is to generalize the idea of isometrically embedding two compact metric spaces into a third one, and of course, this requires a notion of quantum isometry.

\begin{definition}[{\cite{Latremoliere13b}}]
  A \emph{tunnel} $(\D,\Lip,\pi_1,\pi_2)$ from $(\A_1,\Lip_1)$ to $(\A_2,\Lip_2)$ is a {\qcms} $(\D,\Lip)$ and, for each $j\in \{1,2\}$, a *-epimorphism $\pi_j : \D\twoheadrightarrow \A_j$ such that
  \begin{equation*}
    \forall a \in \dom{\Lip_j}, \quad \Lip_j(a) = \inf\left\{ \Lip(b) : b\in\sa{\D}, \pi_j(b) = a  \right\} \text{.}
  \end{equation*}
\end{definition}

We then associate a nonnegative number to every tunnel, which is different from the quantity used in the construction of the Gromov--Hausdorff distance (though still using the Hausdorff distance), in order to accommodate the more general framework of {\qcms s}.

\begin{definition}[{\cite{Latremoliere14}}]\label{extent-def}
  The \emph{extent} $\tunnelextent{\tau}$ of a tunnel $\tau = (\D,\Lip,\pi_1,\pi_2)$ from $(\A_1,\Lip_1)$ to $(\A_2,\Lip_2)$ is the number given by
  \begin{equation*}
    \tunnelextent{\tau} = \max_{j\in\{1,2\}} \Haus{\Kantorovich{\Lip}}\left(\StateSpace(\D), \left\{ \varphi\circ\pi_j : \varphi \in \StateSpace(\A_j) \right\} \right) \text{.}
  \end{equation*}
\end{definition}

The propinquity is thus defined as follows.

\begin{definition}[{\cite{Latremoliere13b,Latremoliere14}}]\label{propinquity-def}
  The \emph{(dual) propinquity} between two {\qcms s} $(\A,\Lip_\A)$ and $(\B,\Lip_\B)$ is given by
  \begin{equation*}
    \dpropinquity{}((\A,\Lip_\A),(\B,\Lip_\B)) = \\
    \inf\left\{ \tunnelextent{\tau} : \text{$\tau$ is a tunnel from $(\A,\Lip_\A)$ to $(\B,\Lip_\B)$} \right\} \text{.}
  \end{equation*}
\end{definition}
We refer to \cite{Latremoliere15b} for the discussion of several important variations of the construction of the propinquity; Definition \ref{propinquity-def} refers to the so-called dual propinquity, which we will simply refer to as the propinquity in this paper.

We will use the following definition in this paper. For a class $C$ and an equivalence relation $\sim$ on $C$, a function $d$ on $C\times C$ is called a \emph{metric up to $\sim$} (or a \emph{pseudo-metric}, in short) if the following three properties hold:
\begin{enumerate}
\item $\forall x,y \in C, \quad d(x,y) = 0 \text{ if and only if }x\sim y$,
\item $\forall x,y \in C, \quad d(x,y) = d(y,x)$,
\item $\forall x,y,z \in C, \quad d(x,z) \leq d(x,y) + d(y,z)$.
  \end{enumerate}
  
\begin{theorem}[{\cite{Latremoliere13b}}]
  The propinquity is a complete metric, up to full quantum isometry, on the class of {\qcms s}. Moreover, it induces the same topology as the Gromov--Hausdorff distance on the class of classical compact metric spaces.
\end{theorem}

\bigskip

It then follows immediately from \cite{Latremoliere13b} that the following result holds.
\begin{theorem}\label{prop-conv-thm}
 If Hypothesis \ref{main-hyp} holds, then
  \begin{equation*}
    \lim_{n\rightarrow\infty} \dpropinquity{}\left( (C(\FC{\infty}),\Lip_{\infty}), (C(\FC{n}),\Lip_n) \right) = 0 \text{.}
  \end{equation*}
\end{theorem}

\begin{proof}
  By \cite{Latremoliere13b}, we have:
  \begin{align*}
    0 &\leq \dpropinquity{}\left( (C(\FC{\infty}),\Lip_{\infty}), (C(\FC{n}),\Lip_n) \right) \\
    &\leq \GH((\FC{n},d_n),(\FC{\infty},d_\infty)) \xrightarrow{n\rightarrow\infty} 0 \text{;}
  \end{align*}
  hence, our conclusion. We note for the record that, in fact, the propinquity restricted to classical compact metric spaces and the Gromov--Hausdorff distance are topologically equivalent.
\end{proof}

\bigskip

Of course, the purpose of introducing the Gromov--Hausdorff propinquity is to give a functional analytic translation of the convergence of $\left(\FC{n}\right)_{n\in\N}$ to $\FC{\infty}$. Next, we actually provide a proof which depends on the specific assumptions about approximation sequences of piecewise $C^1$-fractal curves, rather than just on the more general argument from \cite{Latremoliere13}, since the construction below will be more helpful for our purpose.

\begin{proof}[{Alternate proof of Theorem \ref{prop-conv-thm}}]
  If $f \in C(\FC{m})$, for some $m\in \Nbar$, and if $n \leq m$, then $f|n$ is the restriction of $f$ to $V_n$.
  
  Let $\varepsilon > 0$. By Definition \ref{approx-sequence-def}, there exists $N_0\in\N$ such that if $n\geq N_0$, then:
  \begin{equation*}
    \Haus{d_n}(\FC{n},V_n) < \varepsilon  \text{.}
  \end{equation*}
  By Definition \ref{fractal-curve-def}, there exists $N_1 \in \N$, such that if $n\geq N_1$, then
   \begin{equation*}
    \Haus{d_\infty}(\FC{\infty},V_n) < \varepsilon  \text{.}
  \end{equation*}

  Let $N = \max\{ N_0, N_1 \}$. Fix $n\geq N$ and set $\A_n = C(\FC{\infty}) \oplus C(\FC{n})$.
  
  We now let $\alpha > 0$ as well. For any $(f,g) \in \A_n$, we set
  \begin{equation}\label{M-norm-eq}
    \mathsf{M}_{n,\alpha}(f,g) = \max\left\{ \Lip_{\infty}(f), \Lip_n(g), \frac{1}{\alpha}\norm{f|n - g|n}{C(V_n)} \right\} \text{.}
  \end{equation}

  It is an easy exercise to check that $(\A_n,\mathsf{M}_{n,\alpha})$ is a {\qcms}. We now check that
  \begin{equation}
    \tau_{n,\alpha} = (\A_n,\mathsf{M}_{n,\alpha},\rho_\infty, \rho_n)
  \end{equation}
  is a tunnel, where $\rho_\infty : (f,g) \in \A_n \mapsto f \in C(\FC{\infty})$ and $\rho_n : (f,g) \in \A_n \mapsto g \in C(\FC{n})$. Of course, $\rho_\infty$ and $\rho_n$ are *-epimorphisms.

  If $f \in C(\FC{\infty},\R)$ and $\Lip_{\infty}(f) = 1$, then $f|n$ is also $1$-Lipschitz since $(V_n,d_\infty)$ and $(V_n,d_n)$ are equal, by Definition \ref{approx-sequence-def}. By McShane's extension theorem for real-valued Lipschitz functions \cite{McShane34}, and by Theorem \ref{Lip-is-Lipschitz-thm}, there exists $g \in C(\FC{n},\R)$ such that $g|n = f|n$, and $\Lip_n(g) = 1$. Thus, $\mathsf{M}_{n,\alpha}(f,g) = 1$. In other words, $\Lip_{\infty}$ is the quotient of $\mathsf{M}_{n,\alpha}$ on $C(\FC{\infty})$ via $\rho_\infty$. We caution the reader that, in general, $f$ restricted to $\FC{n}$ is \emph{not} $1$-Lipschitz for $d_n$.

  The same reasoning applies to show that the quotient of $\mathsf{M}_{n,\alpha}$ on $C(\FC{n})$ via $\rho_n$ is equal to $\Lip_n$. Therefore, $\tau_{n,\alpha}$ is a tunnel from $(C(\FC{\infty}),\Lip_\infty)$ to $(C(\FC{n}),\Lip_n)$. Next, we compute its extent.

  \bigskip
  
  First, let $\varphi \in \StateSpace(C(\FC{\infty}))$. For each $x\in \FC{\infty}$, let $\delta_x$ be the Dirac point mass at $x$. By the Krein--Milmann theorem, there exists a finite subset $F\subseteq \FC{\infty}$ such that, if $\theta=\sum_{x\in F} t_x\delta_x$, for some $(t_x)_{x\in F}\in [0,1]^F$ with $\sum_{x\in F}t_x=1$, then
  \begin{equation*}
    \Kantorovich{\Lip_\infty}\left( \varphi, \theta \right) < \alpha \text.
  \end{equation*}
  
  By assumption, for each $x\in F$, there exists $v_x \in V_n$ such that $d_\infty(x,v_x) < \varepsilon$. Let $\psi = \sum_{x\in F} t_x \delta_{v_x}$. By construction, $\psi$ can be trivially identified with a state of $C(V_n)$ and also with a state of $C(\FC{n})$.

  Now, let $(f,g) \in \sa{\A_n}$ be such that $\mathsf{M}_{n,\alpha}(f,g) \leq 1$. Thus, $\Lip_{\infty}(f) \leq 1$, $\Lip_n(g) \leq 1$ and, for all $v \in V_n$, we have $|f(v)-g(v)| < \alpha$. Therefore, we have successively:
  \begin{align*}
    \left|\varphi(f) - \psi(g)\right|
    &= \left|\varphi(f) - \theta(f) \right| + \left|\theta(f) - \psi(f)\right| + \left|\psi(f)-\psi(g)\right| \\
    &\leq \alpha + \sum_{x\in F} t_x |f(x)-f(v_x)| + \sum_{x \in F} t_x |f(v_x) - g(v_x)| \leq 2\alpha + \varepsilon \text{.}
  \end{align*}

  The same reasoning applies with the roles of $\FC{\infty}$ and $\FC{n}$ interchanged. Hence, it follows from Definition \ref{extent-def} that
  \begin{equation*}
    \tunnelextent{\tau_{n,\alpha}}\leq 2\alpha + \varepsilon \text{.}
  \end{equation*}
  Consequently, in light of Definition \ref{propinquity-def}, and since $\alpha>0$ is arbitrary, we have that
  \begin{equation*}
    \dpropinquity{}((C(\FC{\infty}),\Lip_{\infty}), (C(\FC{n}),\Lip_n)) \leq \tunnelextent{\tau_{n,\alpha}} \leq \varepsilon \text.
  \end{equation*}
 This completes our alternative proof of Theorem \ref{prop-conv-thm}.
\end{proof}

The tunnel we constructed in the proof of Theorem \ref{prop-conv-thm} is the main ingredient for obtaining an appropriate estimate on the spectral propinquity between the spectral triples constructed in \cite{Lapidus14} on piecewise $C^1$-fractal curves.

\section{Convergence of Spectral Triples}

\subsection{Spectral Triples for Piecewise $C^1$-Fractal Curves}

Christensen, Ivan and Lapidus defined in \cite{Lapidus08} a metric spectral triple on $C(\SG{\infty})$. Then, Lapidus and Sarhad extended, in \cite{Lapidus14}, the construction from \cite{Lapidus08} to arbitrary piecewise $C^1$-fractal curves.

It is the purpose of this paper to show that the metric spectral triple constructed in \cite{Lapidus08} is a limit of natural metric spectral triples on $\SG{n}$, and more generally, to show that given a compatible approximation sequence for some parametrization of a piecewise $C^1$-fractal curve, the metric spectral triples from \cite{Lapidus14} are limits of spectral triples of finite unions of $C^1$-curves from the chosen parametrization. We now turn to the construction of these spectral triples.

\begin{remark}\label{m-rmk}
  We note that the construction of the spectral triples on the {\SiepG} \cite{Lapidus08}, and, more generally, on piecewise $C^1$-fractal curves \cite{Lapidus14}, provides a noncommutative version (and extension) of the notion of a fractal string, introduced and studied by the second author and his collaborators in, for example, \cite{eighteenhalf} and \cite{thirtyfourhalf,seventeenfourfifth,seventeenhalf}. It would be interesting, in a later work, to establish explicit connections between the present work and the theory of complex dimensions of fractal strings and higher--dimensional fractals developed in those references; see, e.g., \cite{eighteenhalf,seventeenhalf}.
\end{remark}

\bigskip

The construction begins with the construction of a spectral triple on an arbitrary $C^1$-rectifiable curve, much as in \cite{Lapidus08}. First, let us recall the construction of the standard spectral triple on the circle.

Let $\mathrm{CP}$ be the unital Abelian C*-algebra of all $\C$-valued continuous functions $f$ over $[-1,1]$ such that $f(-1) = f(1)$:
\begin{equation*}
  \mathrm{CP} = \left\{ f \in C([-1,1]) : f(-1) = f(1) \right\} \text{.}
\end{equation*}
The Gelfand spectrum of $\mathrm{CP}$ is, of course, homeomorphic to the unit circle in $\C$; we will identify it with the image $\T$ of $[-1,1]$ under the map
\begin{equation*}
  x \in [-1,1] \longmapsto \exp(i\pi x) \text{.}
\end{equation*}

We now define a spectral triple on $\mathrm{CP}$, using the Gelfand--Naimark--Segal representation of $\mathrm{CP}$ for the Haar state. Explicitly, let $\mathscr{J}$ be the Hilbert space closure of $\mathrm{CP}$ for the inner product
\begin{equation*}
  (f,g) \in \mathrm{CP} \mapsto \int_{-1}^1 fg\text{.}
\end{equation*}
As usual, we identify $f \in \mathrm{CP}$ with the (bounded) multiplication operator by $f$ on $\mathscr{J}$.

For each $k \in \Z$, let
\begin{equation*}
  e_k : t \in [-1,1] \mapsto \exp(i \pi k t)\text{.}
\end{equation*}
Clearly, $e_k \in \mathscr{J}$. We define $\slashed{\partial}$ as the closure of the linear extension of the map defined as follows:
\begin{equation*}
  \forall k \in \Z, \quad \slashed{\partial} e_k = \pi k e_k \text{.} 
\end{equation*}
The operator $\slashed{\partial}$ is self-adjoint with spectrum $\left\{ \pi k : k \in \Z \right\}$. In particular, $\slashed{\partial}$ has a compact resolvent.

A quick computation now shows that $f \in \mathscr{J}$ is in the domain $\dom{\slashed{\partial}}$ of $\slashed{\partial}$ if and only if there exists a necessarily unique $g \in \mathscr{J}$ such that
\begin{equation*}
  \forall x \in [-1,1], \quad f(x) = f(0) + \int_0^x g(t) \, dt \text;
\end{equation*}
i.e., $f$ is absolutely continuous on $[-1,1]$, with almost everywhere derivative $g$. Furthermore, in this case, $\slashed{\partial} f = ig$. From this, it follows that for all $k\in\Z$, we have $[\slashed{\partial},f]e_k = (\slashed{\partial} f) e_k$. We thus deduce that, if we let
\begin{equation*}
  \Lip_\T : f \in \mathrm{CP} \mapsto \opnorm{[\slashed{\partial},f]}{}{\mathscr{J}} \quad \text{ (allowing for the value $\infty$),}
\end{equation*}
then
\begin{equation*}
  \forall f \in \mathrm{CP}, \quad \Lip_\T(f) = \norm{\slashed{\partial} f}{L^\infty([-1,1])} \quad \text{ (also allowing for the value $\infty$).}
\end{equation*}

From this, we conclude that $f \in \dom{\Lip_{\T}}$ if and only if $\slashed{\partial}(f)$ is essentially bounded on $[-1,1]$. Equivalently, via the Lebesgue differentiation theorem, $f \in \dom{\Lip_\T}$ if and only if $f$ is Lipschitz for the \emph{usual metric} on $[-1,1]$ --- with the obvious identification of $\mathscr{J}$ as a closed subspace of $L^2([-1,1])$. In turn, this implies that
\begin{equation*}
  \left(\mathrm{CP},\mathscr{J},\slashed{\partial}\right)\text{ is a metric spectral triple}
\end{equation*}
since $\{ f \in \mathrm{CP} : f(0) = 0, \Lip_{\T}(f) \leq 1 \}$ is compact in $C([-1,1])$ by Arz{\'e}la--Ascoli theorem. However, we want to understand the metric induced by $\Lip_\T$ on the Gelfand spectrum $\T$ of the C*-algebra $\mathrm{CP}$. Let $x,y \in [0,1)$. It is easy to see that
\begin{multline*}
  \Kantorovich{\Lip_\T}\left( \exp(2i\pi x),\exp(2i\pi y) \right) = \\ \sup\left\{ |f(x)-f(y)| : f\in \mathrm{CP}, f(1) = 0, \norm{\slashed{\partial}(f)}{L^\infty([0,1])} \leq 1 \right\}\text{.}
\end{multline*}
If $f \in \mathrm{CP}$ with $\Lip_\T(f)\leq 1$ and $f(1) = 0$, and thus $f(-1) = 0$, then
\begin{equation*}
  |f(x)-f(y)| \leq \min\{ |x-y|, 2-|x-y| \}
\end{equation*}
and therefore,
\begin{equation}\label{Kantorovich-computation-eq}
  \forall x,y \in [-1,1], \quad \Kantorovich{\Lip_\T}(\exp(2i\pi x),\exp(2i\pi y))  = |x-y| \quad \left(\mathrm{mod} \; 1 \right) \text{,}
\end{equation}
where equality in Equation (\ref{Kantorovich-computation-eq}) is achieved by using continuous piecewise affine functions.

Thus, the metric induced by $\Kantorovich{\Lip_\T}$ makes the Gelfand spectrum of $\mathrm{CP}$ isometric to the unit circle in $\T$ endowed with its geodesic distance; i.e., the distance between two distinct points is the smallest of the lengths of the two arcs between these points.

\bigskip

We now use the spectral triple $(\mathrm{CP},\mathscr{J},\slashed{\partial})$ in order to construct a spectral triple on the unit interval $[0,1]$. As the construction of spectral triples on piecewise $C^1$-fractal curves involves possibly countable direct sums of interval spectral triples, we will in particular avoid having the eigenvalue 0 in the spectrum of our interval Dirac operator, so that a countable direct sum of such operators will still have a compact resolvent.

If $f \in C([0,1])$, then the map $t \in [-1,1]\mapsto f(|t|)$ is in $CP$. Let $\varpi$ be the faithful *-representation of $C([0,1])$ on $\mathscr{J}$ defined by
\begin{equation*}
  \forall f \in C([0,1]), \quad \forall \xi \in \mathscr{J}, \quad \varpi(f)\xi : t \in [-1,1] \mapsto f(|t|)\xi(t) \text{.}
\end{equation*}
We also set $\Dirac = \slashed{\partial} + \frac{\pi}{2}$ and $\dom{\Dirac}=\dom{\slashed{\partial}}$, noting that $\Dirac$ is a self-adjoint operator with spectrum $\spectrum{\Dirac} = \left\{ \pi\left(k + \frac{1}{2}\right) : k \in \Z \right\}$. It is easy to check that $\left( C([0,1]), \mathscr{J}, \Dirac \right)$ is a metric spectral triple over $C([0,1])$ which induces the usual metric on $[0,1]$.

\bigskip

For each $n\in\Nbar$, we now construct our spectral triple over $\FC{n}$, where we use Hypothesis \ref{main-hyp}. We let
\begin{equation*}
  \Hilbert_n = \oplus_{j = 0}^{B_n} \mathscr{J}
\end{equation*}
and
\begin{equation*}
  \dom{D_n} = \left\{ (\xi_j)_{j=0}^{B_n} \in \Hilbert_n : \forall j\in\{0,\ldots,B_n\} \quad \xi_j \in \dom{\Dirac} \right\} \text{.}
\end{equation*}

For each $j \in \N$, we also let $q_j : C(\FC{\infty}) \twoheadrightarrow C([0,1])$, which sends $f \in C(\FC{\infty})$ to $f \circ C_j$ in $C[0,1]$. Of course, $q_j$ is a *-epimorphism.  We then set, for all $f \in C(\FC{n})$:
\begin{equation*}
  \forall \xi = (\xi_j)_{j \in\N, j \leq B_n} \in \Hilbert_n, \quad \pi_n(f)\xi = \left( \varpi(f\circ C_j)\xi_j \right)_{j\in\N, j \leq B_n}\text{.}
\end{equation*}
Finally, using the same notation as above, we set:
\begin{equation*}
   \forall \xi = (\xi_j)_{j \in\N, j \leq B_n} \in \dom{D_n}, \quad D_n \xi = \left( \frac{1}{\lambda_j} \Dirac\xi_j \right)_{j \in \N, j\leq B_n} \text{,}
 \end{equation*}
 where $\lambda_j$ is the length of $C_j$, for every $j\in\N$.

It is then easily checked that $\left( C(\FC{n}), \Hilbert_n, D_n \right)$ is a spectral triple on $C(\FC{n})$. We will next show that this spectral triple is metric and that $\Kantorovich{D_n}$ restricted to $\FC{n}$ coincides with the geodesic distance $d_n$. Our theorem includes \cite[Theorem 8.13]{Lapidus08} (case $n=\infty$) and extends it to all $n\in\Nbar$, which we need in order to be able to formulate and establish our approximation results.

\begin{theorem}\label{Lip-is-Lipschitz-thm}
  We assume Hypothesis \ref{main-hyp}. Let $n\in\Nbar$. If $f \in \FC{n}$, then
  \begin{equation*}
    f \in \dom{\Lip_n} \iff f\dom{D_n}\subseteq \dom{D_n}
  \end{equation*}
  and, for all $f\in\dom{\Lip_n}$, we have
  \begin{equation*}
    \Lip_n(f) = \opnorm{[D_n,\pi_n(f)]}{}{\Hilbert_n} \text{.}
  \end{equation*}
  In particular, the restriction of $\Kantorovich{D_n}$ to $\FC{n}$ is the geodesic distance $d_n$.
\end{theorem}

\begin{remark}
  It is not sufficient to show that the restriction of $\Kantorovich{D_n}$ is $d_n$ in order to conclude that Theorem \ref{Lip-is-Lipschitz-thm} holds --- see, for instance, \cite{Latremoliere15-0}, where two different L-seminorms on the continuous functions over the Cantor set give the same metric on the Cantor set but \emph{not} on the state space.
\end{remark}

\begin{proof}[Proof of Theorem \ref{Lip-is-Lipschitz-thm}]
Fix $n\in\Nbar$. By construction, since for all $f \in C(\FC{n})$, such that $f\dom{D_n}\subseteq\dom{D_n}$,
\begin{equation*}
  \opnorm{\left[D_n,f\right]}{}{\Hilbert_n} = \sup_{j\in\N,j\leq B_n} \opnorm{\left[\frac{1}{\lambda_j}\Dirac,f \right]}{}{\mathscr{J}}\text{,}
\end{equation*}
we conclude that for all $k\geq 0$, the following two assertions are equivalent:
\begin{itemize}
\item $\opnorm{\left[D_n,f\right]}{}{\Hilbert_n} \leq k$,
\item for all $j\in \{0,\ldots,B_n\}$, the restriction of $f$ to the curve $C_j$ is $k$-Lipschitz for $d_n$.
\end{itemize}

We first work with $n\in\N$.

First, let $f \in C(\FC{n})$ such that $\opnorm{\left[D_n,\pi_n(f)\right]}{}{\Hilbert_n} \leq 1$. 

Furthermore, let $x,y \in C_j$ for $j \in \{0, \ldots, B_n\}$. Finally, let  $t,t'\in [0,1]$ such that $C_j(t) = x$ and $C_j(t') = y$. We then conclude that
\begin{equation*}
  |f\circ C_j(t') - f\circ C_j(t)| \leq \lambda_j | t - t' | = d_n(C_j(t),C_j(t')) = d_n(x,y)\text.
\end{equation*}
Indeed, $\opnorm{[\Dirac,\varpi(f\circ C_j)]}{}{\mathscr{J}} \leq \lambda_j$ and $C_j$ is a geodesic in $\FC{\infty}$ and thus also in $\FC{n}$. 

Now, let $x,y \in \FC{n}$. Let $\gamma$ be a geodesic from $x$ to $y$ in $\FC{n}$, and recall that $\gamma$ is injective. Let $\lambda=\mathrm{length}(\gamma)=d_n(x,y)$. Since $V_n$ is finite, there exists a finite set $F \subseteq [0,1]$ such that $\gamma(t) \in V \iff t \in F$. Write $F = \{ t_1,\ldots,t_r \}$, with $t_1 < \ldots < t_r$. Let $t_0$ be such that $\gamma(t_0) = x$ and let $t_{r+1}$ be such that $\gamma(t_{r+1}) = y$.

Moreover, let $j\in\{0,\ldots,r\}$. Since, by construction, $\gamma|(t_j,t_{j+1})$ does not contain a vertex, it follows by continuity and the intermediate value theorem that $\gamma(t_j)$ and $\gamma(t_{j+1})$ belong to the same curve $C_h$, for some $h\in\{0,\ldots,B_n\}$. Thus, $|f(\gamma(t_{j+1}))-f(\gamma(t_j))|\leq d_n(\gamma(t_j),\gamma(t_{j+1}))$. We now have successively:
\begin{align*}
  |f(x) - f(y)|
  &\leq \sum_{j=0}^r |f(\gamma(t_{j+1}))-f(\gamma(t_j))| \\
  &\leq \sum_{j=0}^r d_n(\gamma(t_{j}),\gamma(t_{j+1}))  \\
  &= \sum_{j=0}^r \lambda |t_{k+1} - t_k| = \lambda \sum_{j=0}^r (t_{k+1}-t_k) = \lambda (t_{r+1}-t_0) \\
  &= d_n(x,y) \text{.}
\end{align*}

Hence, we conclude that for all $x,y \in \FC{n}$, we have
\begin{equation*}
  |f(x)-f(y)|\leq d_n(x,y)\text{.}
\end{equation*}
Since $x,y \in \FC{n}$ were arbitrary, we have shown that $\Lip_n(f)\leq 1$. By homogeneity, it follows that for all $f \in C(\FC{n})$, we have
\begin{equation*}
  \Lip_n(f) \leq \opnorm{[D_n,\pi_n(f)]}{}{\Hilbert_n}\text{.}
\end{equation*}

\bigskip

Now, let $f \in C(\FC{n})$ with $\Lip_n(f)\leq 1$. Furthermore, let $j \in \{0,\ldots,B_n\}$, and let $x,y \in C_j$. By definition, there exists $t,t' \in [0,1]$ with $C_j(t)=x$ and $C_j(t')=y$. We then have that
\begin{align*}
  \left| f(x) - f(y) \right|
  &= \left| f\circ C_j(t) - f\circ C_j(t') \right| \\
  &\leq d_n(C_j(t),C_j(t')) \leq \lambda_j | t - t' | \text{.}
\end{align*}
Thus, $f\circ C_j$ is a $\lambda_j$-Lipschitz on $[0,1]$, and therefore, $\opnorm{[\Dirac,\varpi(f\circ C_j)]}{}{\mathscr{J}} \leq \lambda_j$. Hence, $\opnorm{\left[\frac{1}{\lambda_j}\Dirac,\varpi(f\circ C_j)\right]}{}{\mathscr{J}} \leq 1$. Since $j$ is arbitrary in $\{0,\ldots,B_n\}$, we conclude that
\begin{equation*}
  \opnorm{\left[D_n,\pi_n(f)\right]}{}{\Hilbert_n} = \sup_{j\in\N,j\leq B_n}\opnorm{\left[\frac{1}{\lambda_j}\Dirac, \varpi(f\circ C_j)\right]}{}{\mathscr{J}} \leq 1 \text{.}
\end{equation*}

Consequently, for all $f \in C(\SG{n})$, we have shown, as desired, that
\begin{equation*}
  \Lip_n(f) = \opnorm{[D_n,\pi_n(f)]}{}{\Hilbert_n} \text{.}
\end{equation*}

For $n=\infty$, we can proceed as follows.

First, let $f \in C(\HG{\infty})$ with $\opnorm{[D_\infty,\pi_\infty(f)]}{}{\Hilbert_\infty} \leq 1$. Thus, for all $n\in\N$, we then have that $\opnorm{[D_n,\pi_n(f)]}{}{\Hilbert_n}\leq 1$.

Let $v,w \in V_\infty$, therefore $v,w \in V_N$ for some $N$. We have seen that $d_\infty(v,w) = d_n(v,w)$. If $\Lip_\infty(f) \leq 1$, then $\Lip_n(f)\leq 1$ and, as above, $|f(v)-f(y)|\leq d_n(v,w) = d_\infty(v,w)$. By continuity, since $V_\infty$ is dense in $\HG{\infty}$, we conclude that $|f(x)-f(y)|\leq d_\infty(x,y)$, for all $x,y \in \FC{\infty}$. Hence, $\Lip_\infty(f) \leq 1$.

If $f$ is $1$-Lipschitz on $(\HG{\infty},d_\infty)$, then, for any $j\in\N$, its restriction to $C_j$, is $1$--Lipschitz on $(C_j,d_j)$ --- indeed, $d_\infty\leq d_j$ so $\mathrm{Lip}_j(f)\leq \mathrm{Lip}_\infty(f) \leq 1$. Thus, $\opnorm{\left[\frac{1}{\lambda_j}\Dirac,\varpi(f\circ C_j)\right]}{}{} \leq 1$. By construction, it follows that $\Lip_\infty(f) \leq 1$.

By homogeneity, we conclude that $\opnorm{[D_\infty,\pi_\infty(f)]}{}{\Hilbert_\infty} = \Lip_\infty(f)$, for all $f \in C(\FC{\infty})$. This completes the proof of our theorem.
\end{proof}

Spectral triples contain other geometric data besides the metric information, such as how to recover from it the Hausdorff measure and the Hausdorff dimension of $\SG{\infty}$, via a Dixmier trace construction. We refer to \cite{Lapidus08} for some of these properties in the case of $C(\SG{\infty},\Hilbert_\infty,D_\infty)$. (See also, e.g., \cite{twoquarter, Connes, sixfourfifth, fourteenfourfifth, Lapidus94,  Lapidus97} for related results in various contexts.)

\bigskip

We conclude this section by adopting several additional conventions. If $n\leq m$, we then identify $\Hilbert_n$ with a subspace of $\Hilbert_m$ via the linear embedding
\begin{equation*}
  \left(\xi_{j}\right)_{j\in\N,j\leq B_n} \longmapsto \left( \begin{cases} \xi_{j} \text{, if $j \leq B_n$} \\ 0 \text{, otherwise} \end{cases} \right)_{j\in\N,j\leq B_m} \text{.}
\end{equation*}
Moreover, if $f \in C(\FC{m})$ and $\xi \in \Hilbert_n$, we then write $f\xi$ for $\pi_n(h)\xi$, where $h$ is the restriction of $f$ to $\FC{n}$. In particular, we will dispense with the notation $\pi_n$ in the remainder of this paper whenever no confusion may arise.

\subsection{Modular Convergence}

Let $(\A,\Hilbert,D)$ be a metric spectral triple. By \cite{Latremoliere18g}, if we define $\CDN$ and $\Lip_\D$ as follows,
\begin{equation*}
  \forall \xi \in \dom{D}, \quad \CDN(\xi) = \norm{\xi}{\Hilbert} + \norm{D\xi}{\Hilbert}
\end{equation*}
where $\dom{D}$ is the domain of $D$, and
\begin{equation*}
  \forall a \in \dom{\Lip_D}, \quad \Lip_D(a) = \opnorm{[D,a]}{}{\Hilbert},
\end{equation*}
where
\begin{equation*}
  \dom{\Lip_D} = \left\{ a \in \sa{\A} : a\dom{D}\subseteq\dom{D},[D,a]\text{ is bounded} \right\}\text,
\end{equation*}
then the tuple
\begin{equation}\label{qvb-def-eq}
  \qvb{\A}{\Hilbert}{D} = (\Hilbert,\CDN,\C,0,\A,\Lip_D)
\end{equation}
is an example of a \emph{quantum metrical vector bundle}, in the following sense.

\begin{definition}[{\cite{Latremoliere16c,Latremoliere18d}}]\label{qvb-def}
  A \emph{quantum metrical vector bundle}
  \begin{equation*}
    (\module{M},\CDN,\B,\Lip_\B,\A,\Lip_\A)
  \end{equation*}
  is given by two {\qcms s} $(\A,\Lip_\A)$ and $(\B,\Lip_\B)$, a right Hilbert $\B$-module which also carries a left $\A$-module structure, and a norm $\CDN$ defined on a dense $\A$-submodule, $\dom{\CDN}$, of $\module{M}$ such that the following properties hold:
  \begin{enumerate}
  \item $\forall \omega \in \dom{\CDN}, \quad \norm{\omega}{\module{M}} \leq \CDN(\omega)$,
  \item $\{ \omega\in\dom{\CDN} : \CDN(\omega) \leq 1 \}$ is compact in $\norm{\cdot}{\module{M}}$,
  \item for all $\omega,\eta\in\module{M}$, denoting $b = \inner{\omega}{\eta}{\module{M}} \in \B$, we have
    \begin{equation*}
      \max\left\{ \Lip_\B\left(\frac{b+b^\ast}{2}\right), \Lip_\B\left(\frac{b-b^\ast}{2i}\right)\right\} \leq 2 \CDN(\omega)\CDN(\eta)
    \end{equation*}
    (we refer to this inequality as the \emph{inner Leibniz inequality}),
  \item for all $a\in\dom{\Lip_\A}$ and $\omega\in \dom{\CDN}$, we have
    \begin{equation*}
      \CDN(a\omega) \leq \left(\norm{a}{\A} + \Lip_\A(a)\right)\CDN(\omega)\text{.}
    \end{equation*}
    (we refer to this inequality as the \emph{modular Leibniz inequality}).
  \end{enumerate}
  The norm $\CDN$ is called a \emph{D-norm}.
  
  When $(\module{M},\CDN,\B,\Lip_\B,\C,0)$ is a metrical quantum vector bundle, the tuple
  \begin{equation*}
    (\module{M}, \CDN,\B,\Lip_\B)
  \end{equation*}
  is called a \emph{\gQVB}. 
\end{definition}

\begin{remark}
  When $\Hilbert$ is a Hilbert space, it can be seen as a right module over $\C$ trivially, by setting $\omega\cdot z = z\omega$ for all $z\in\C$ and $\omega\in\Hilbert$, since $\C$ is Abelian. We will typically write scalars as usual  on the left when working with Hilbert spaces, while implicitely considering them as right modules when part of a metrical quantum vector bundle.
\end{remark}

While in \cite{Latremoliere18d}, more general forms of the modular and inner Leibniz inequalities are allowed, Definition \ref{qvb-def} is the important special case which we use when working with spectral triples.

We use the following natural notions of morphisms between Hilbert modules, which will underlie the various notions of isomorphisms for quantum metrical vector bundles.

\begin{definition}
  Let $\A$,$\B$ be two unital C*-algebras. A \emph{left module morphism} $(\Pi,\pi)$ from a left $\A$-module $\module{M}$ to a lft $\B$-module $\module{N}$ is given by the following data:
  \begin{itemize}
  \item a unital *-morphism $\pi : \A\rightarrow\B$,
  \item a linear map $\Pi : \module{M}\rightarrow\module{N}$ such that
    \begin{equation*}
      \forall a \in \A, \quad \forall \omega \in \module{M}, \quad \Pi(a \omega) = \pi(a)\Pi(\omega) \text{.}
    \end{equation*}
  \end{itemize}

  The module morphism $(\Pi,\pi)$ is said to be \emph{surjective} when both $\Pi$ and $\pi$ are surjective maps, and it is said to be an \emph{isomorphism} when both $\Pi$ and $\pi$ are bijections.

  A \emph{right module morphism} is defined similarly.
  
  A Hilbert module morphism $(\Pi,\pi)$ from a right $\A$-Hilbert module $\module{M}$ to a right $\B$-Hilbert module $\module{N}$ is a module morphism when
  \begin{equation*}
    \forall \omega,\xi \in \module{M}, \quad \inner{\Pi(\omega)}{\Pi(\xi)}{\module{N}} = \inner{\omega}{\xi}{\module{M}} \text{.}
  \end{equation*}
\end{definition}

We refer to \cite{Latremoliere16c} for examples of such a structure and for its motivations.

\bigskip

Our next step in proving the convergence of $(C(\FC{n}),\Hilbert_n,D_n)$ to $(C(\FC{\infty}),\Hilbert_\infty,D_\infty)$ is to establish the convergence of their associated metrical quantum vector bundles.

\begin{notation}
To begin with, for each $n\in\Nbar$, we consider $\Hilbert_n$ as a $\C$-Hilbert module endowed with the following D-norm:
\begin{equation*}
  \forall \xi \in \dom{D_n}, \quad \CDN_n(\xi) = \norm{\xi}{\Hilbert_n} + \norm{D_n\xi}{\Hilbert_n} \text{.}
\end{equation*}
Thus, by \cite{Latremoliere18g}, $(\Hilbert_n,\CDN_n,\C,0)$ is a {\gQVB}. We next compute the modular propinquity between $(\Hilbert_n,\CDN_n,\C,0)$ and $(\Hilbert_\infty,\CDN_\infty,\C,0)$.
\end{notation}

The third author defined the modular propinquity in \cite{Latremoliere16c,Latremoliere18d} by extending the notion of a tunnel between {\qcms s} to the notion of a tunnel between {\gQVB s}.

\begin{definition}[{\cite{Latremoliere18d}}]
  Let $(\module{M}_j,\CDN^j,\A_j,\Lip_j)$ be a {\gQVB}, for $j\in\{1,2\}$. A \emph{modular tunnel} $(\mathds{D},(\Pi_1,\pi_1),(\Pi_2,\pi_2))$ is given by
  \begin{enumerate}
  \item a {\gQVB} $\mathds{D} = (\module{P}, \CDN', \D, \Lip_\D)$,
  \item a tunnel $(\D,\Lip_\D,\pi_1,\pi_2)$ from $(\A_1,\Lip_1)$ to $(\A_2,\Lip_2)$,
  \item for each $j\in\{1,2\}$, $(\Pi_j,\pi_j)$ is a surjective Hilbert module morphism from $\module{P}$ (over $\A_1$) to $\module{M}_j$ such that
    \begin{equation*}
      \forall \omega \in \module{M}_j, \quad \CDN^j(\omega) = \inf\left\{ \CDN'(\eta) : \Pi_j(\eta) = \omega \right\} \text{.}
    \end{equation*}
  \end{enumerate}
\end{definition}

The extent of a modular tunnel is computed just like the extent of the underlying tunnel between {\qcms s}, as we see in the following definition.

\begin{definition}[{\cite{Latremoliere18d}}]
  The \emph{extent}, $\tunnelextent{\tau}$, of a modular tunnel $\tau = (\mathds{D},(\Pi_1,\pi_1),(\Pi_2,\pi_2))$, with $\mathds{D} = (\module{P},\CDN',\D,\Lip_\D)$, is the extent of the tunnel $(\D,\Lip_\D,\pi_1,\pi_2)$.
\end{definition}

The modular propinquity is then defined along the same lines as the propinquity.

\begin{definition}[{\cite{Latremoliere16c,Latremoliere18d}}]
  The \emph{modular propinquity}, $\modpropinquity{}(\mathds{A},\mathds{B})$, between two {\gQVB s} $\mathds{A}$ and $\mathds{B}$ is the nonnegative number given by
  \begin{equation*}
    \dmodpropinquity{}(\mathds{A},\mathds{B}) = \inf\left\{\tunnelextent{\tau} : \text{ $\tau$ is a modular tunnel from $\mathds{A}$ to $\mathds{B}$} \right\}\text{.}
  \end{equation*}
\end{definition}

We now record a few fundamental properties of the modular propinquity.

\begin{theorem}[{\cite{Latremoliere16c,Latremoliere18d}}]
  The modular propinquity is a complete metric, up to the equivalence according to which two {\gQVB s} $(\module{M},\CDN,\A,\Lip_\A)$ and $(\module{N},\CDN',\B,\Lip_\B)$ are fully isometrically isomorphic if and only if there exists a Hilbert module isomorphism $(\Pi,\pi)$ such that
  \begin{itemize}
  \item $\Lip_\B\circ\pi = \Lip_\A$,
  \item $\CDN'\circ\Pi = \CDN$.
  \end{itemize}
\end{theorem}

Note that in the definition of a tunnel, we allow --- as we must --- Hilbert modules over C*-algebras which are not necessarily equal to $\C$. In fact, if we restricted tunnels between Hilbert spaces to only involve $\C$ as a C*-algebra, then their extent would always be null, but this is obvious since such a tunnel can only be between full isometric {\gQVB s}.

We next construct our modular tunnels and obtain an estimate on the modular propinquity between $(\Hilbert_n,\CDN_n,\C,0)$ and $(\Hilbert_\infty,\CDN_\infty,\C,0)$, for all $n\in\N$.

\begin{lemma}\label{modular-conv-lemma}
  Assume Hypothesis \ref{main-hyp}. Then, the following limit holds:
  \begin{equation*}
    \lim_{n\rightarrow\infty} \dmodpropinquity{}\left( (\Hilbert_n,\CDN_n,\C,0), (\Hilbert_\infty,\CDN_\infty,\C,0) \right) = 0 \text{.}
  \end{equation*}
\end{lemma}

\begin{proof}
Let $\varepsilon > 0$. There exists $N \in \N$ such that if $j\geq N$, then $\lambda_j < \frac{\pi \varepsilon}{2}$, by Condition (1) of Definition \ref{fractal-curve-def}.

Let $\xi=(\xi_j)_{j\in\N,j\leq B_n} \in \dom{D_\infty}$. Using the Hilbert basis $(e_k)_{k\in\Z}$ of $\mathscr{J}$ of eigenvectors of $\Dirac$, we write $\xi_{j} = \sum_{k\in\Z} t_{j,k} e_k$, for $(t_{j,k})_{k\in\Z} \in \ell^2(\Z)$ and for all $j\in\N$.

We now obtain the following lower bound:
\begin{align*}
  \norm{D_\infty\xi}{\Hilbert_\infty}^2
  &= \sum_{j\in\N} \norm{\frac{1}{\lambda_j^2} \Dirac \xi_{j}}{\mathscr{J}}^2\\
  &= \sum_{j\in\N} \sum_{k\in\Z} \frac{\pi^2}{\lambda_j^2}\left(k+\frac{1}{2}\right)^2 |t_{j,k}|^2 \\
  &\geq \sum_{j\geq N} \sum_{k\in\Z} \frac{\pi^2}{\lambda_j^2}\left(k+\frac{1}{2}\right)^2|t_{j,k}|^2 \text{.}
\end{align*}

If $\CDN_\infty(\xi) \leq 1$, then we conclude from the above expression that
\begin{equation*}
\sum_{j \geq N} \sum_{k\in\Z} |t_{j,k}|^2 \frac{\pi^2\left(k+\frac{1}{2}\right)^2}{\frac{\pi^2 \varepsilon^2}{4}} \leq \sum_{j \geq N} \sum_{k\in\Z} |t_{j,k}|^2 \frac{\pi^2\left(k+\frac{1}{2}\right)^2}{\lambda_j^2} \leq 1\text.
\end{equation*}
Thus, if $\CDN_\infty(\xi)\leq 1$, and since $(k+1)^2 \geq \frac{1}{4}$ for all $k\in\Z$, we deduce that
\begin{align*}
  \sum_{j \geq N+1}\norm{\xi_{j}}{\mathscr{J}}^2
  &= \sum_{j \geq N+1} \sum_{k\in\Z}|t_{j,k}|^2 \\
  &\leq 4 \sum_{n\geq N+1}\sum_{k\in\Z}\left(k+\frac{1}{2}\right)^2\pi^2|t_{j,k}|^2 \\
  &< \varepsilon^2 \text{.}
\end{align*}

Let $n\geq N+1$. We set, for $(\xi,\eta)\in\Hilbert_\infty\oplus\Hilbert_n$,
\begin{equation*}
  \mathsf{T}_n(\xi,\eta) = \max\left\{ \CDN_\infty(\xi), \CDN_n(\eta), \frac{1}{\varepsilon}\norm{\xi-\eta}{\Hilbert_\infty} \right\}\text{.}
\end{equation*}

Let $\Pi_\infty : (\xi,\eta) \in \Hilbert_\infty\oplus\Hilbert_n \mapsto \xi \in \Hilbert_\infty$ and $\Pi_n : (\xi,\eta) \in \Hilbert_\infty\oplus\Hilbert_n \mapsto \eta \in \Hilbert_n$. We begin by proving that the quotient of $\mathsf{T}_n$ for $\Pi_\infty$ is $\CDN_\infty$ and the quotient of $\mathsf{T}_n$ for $\Pi_n$ is $\CDN_n$.

\bigskip

If $\eta\in\Hilbert_{n}$ with $\CDN_n(\eta)\leq 1$, we then let $\xi = \eta$. It is immediate, by definition, that $\CDN_\infty(\eta) = \CDN_n(\eta) \leq 1$ and $\norm{\eta-\xi}{\Hilbert_\infty} = 0$; so that $\mathsf{T}_n(\xi,\eta) \leq 1$.

If $\xi \in \Hilbert_\infty$ with $\CDN_\infty(\xi) \leq 1$, we then let $\eta$ be the orthogonal projection of $\xi$ onto $\Hilbert_n$. Then, again by construction, $\CDN_n(\eta)\leq \CDN_\infty(\xi) \leq 1$. Moreover, our choice for $N$ guarantees that $\norm{\xi-\eta}{\Hilbert_\infty} < \varepsilon$; so that $\mathsf{T}_n(\xi,\eta)\leq 1$. Hence, as claimed, $\mathsf{T}_n$ quotients to both $\CDN_\infty$ and $\CDN_n$.

\bigskip

We now prove that $\mathsf{T}_n$ is a D-norm for $\module{M} = \Hilbert_\infty \oplus \Hilbert_n$, seen as a $\C\oplus\C$-Hilbert module via the following action and inner product:
\begin{equation*}
  \forall (z,w)\in \C\oplus\C, \quad \forall (\xi,\eta)\in \module{M}, \quad (\xi,\eta)\cdot(z,w) = (z\xi,w\eta)
\end{equation*}
and
\begin{equation*}
  \forall (\xi,\eta),(\xi',\eta')\in\module{M}, \quad \inner{(\xi,\eta)}{(\xi',\eta')}{\module{M}} = \left( \inner{\xi}{\xi'}{\Hilbert_\infty}, \inner{\eta}{\eta'}{\Hilbert_n} \right) \text{.}
\end{equation*}
To this end, we endow $\C\oplus \C$ with the L-seminorm $\mathsf{Q}$ given by 
\begin{equation*}
  \forall (z,w) \in \C \oplus \C\text, \quad Q(z,w) = \frac{1}{\varepsilon}|z-w| \text{.}
\end{equation*}
It is immediate that $(\C\oplus\C,Q)$ is a {\qcms}.

Let us turn to several needed properties of $\mathsf{T}_n$. As the maximum of three lower semicontinuous functions, $\mathsf{T}_n$ is lower semicontinuous as well. Hence, its unit ball is closed. We also easily see that
\begin{equation*}
  \forall (\xi,\eta)\in\Hilbert_n\oplus\Hilbert_\infty, \quad \mathsf{T}_n(\xi,\eta)\geq \max\{\norm{\xi}{\Hilbert_n},\norm{\eta}{\Hilbert_\infty} \} = \norm{(\xi,\eta)}{\Hilbert_n \oplus \Hilbert_\infty} \text{.}
\end{equation*}

Furthermore,
\begin{multline*}
  \left\{(\xi,\eta)\in\module{M} :\mathsf{T}_n(\xi,\eta)\leq 1\right\}  \\
  \subseteq \left\{\xi\in\Hilbert_\infty: \CDN_\infty(\xi) \leq 1\right\} \times \left\{ \eta\in\Hilbert_n : \CDN_n(\eta) \leq 1 \right\}\text{;}
\end{multline*}
so that the unit ball of $\mathsf{T}_n$ is closed in a compact set, and hence is compact.

Moreover, 
\begin{align*}
  Q\left(\inner{(\xi,\eta)}{(\xi',\eta')}{\module{M}}\right)
  &= \frac{1}{\varepsilon}\left|\inner{\xi}{\xi'}{\Hilbert_\infty} - \inner{\eta}{\eta'}{\Hilbert_n} \right|\\
  &= \frac{1}{\varepsilon}\left|\inner{\xi}{\xi'}{\Hilbert_\infty} - \inner{\eta}{\eta'}{\Hilbert_\infty} \right|\\
  &\leq \frac{1}{\varepsilon}\left(\left|\inner{\xi}{\xi'-\eta'}{\Hilbert_\infty}\right| + \left|\inner{\xi-\eta}{\eta'}{\Hilbert_\infty}\right|\right)\\
  &\leq \norm{\xi}{\Hilbert_\infty}\mathsf{T}_n(\xi',\eta') + \norm{\eta'}{\Hilbert_n}\mathsf{T}_n(\xi,\eta) \\
  &\leq 2 \mathsf{T}_n(\xi,\eta) \mathsf{T}_n(\xi',\eta') \text{.}
\end{align*}

Hence, the inner Leibniz property holds as well; see Assertion (3) of Definition \ref{qvb-def}.

We thus have proven that $\mu_n = \left(\Hilbert_n\oplus\Hilbert_\infty,\mathsf{T}_n,\C\oplus\C,Q\right)$ is a modular tunnel from $(\Hilbert_n,\CDN_n,\C,0)$ to $(\Hilbert_\infty,\CDN_\infty,\C,0)$.

It is immediate to compute that the extent of $(\C\oplus\C,Q)$ is equal to $\varepsilon$. Thus, the lemma is proven.
\end{proof}

We remark that, of course, $\dpropinquity{}((\C,0),(\C,0)) = 0$. However, to give an idea of what occurs in the previous proof, note that if we replace $Q$ with $\delta Q$ for $\delta > 0$ but very small, then $\mathsf{T}_n$ will not satisfy the Leibniz identity. Therefore, the Leibniz condition is what enforces the rigidity which then makes the distance between $(\Hilbert_\infty,\CDN_\infty)$ and $(\Hilbert_n,\CDN_n)$ nonzero.

\bigskip

We next compute how far are the actions of $C(\FC{\infty})$ on $\Hilbert_\infty$ and of $C(\FC{n})$ on $\Hilbert_n$. This is accomplished by bringing together the tunnels from the proof of Theorem \ref{prop-conv-thm} and from Lemma \ref{modular-conv-lemma}. The only thing left to check is another form of the Leibniz property. Indeed, the modular propinquity can be easily extended to metrical quantum vector bundles, as we now explain.

\begin{definition}[{\cite{Latremoliere18d}}]\label{metrical-tunnel-def}
  Let $\mathds{A}^j = (\module{M}_j,\CDN_{\module{M}_j},\A^j,\Lip^j,\B^j,\Lip_j)$, for $j\in\{1,2\}$.

  A \emph{metrical tunnel} $(\tau,\tau')$ from $\mathds{A}^1$ to $\mathds{A}^2$ is given by the following data:
   \begin{enumerate}
     \item a modular tunnel $\tau = (\mathds{D},(\theta_1,\Theta_1),(\theta_2,\Theta_2))$ from $(\module{M}_1,\CDN_{\module{M}_1},\A^1,\Lip^1)$ to $(\module{M}_2,\CDN_{\module{M}_2},\A^2,\Lip^2)$, where we write $\mathds{D} = (\module{P},\CDN,\D,\Lip_\D)$,
     \item a tunnel $\tau' = (\D',\Lip',\pi^1,\pi^2)$ from $(\B^1,\Lip_1)$ to $(\B^2,\Lip_2)$,
     \item $\module{P}$ is also a $\D'$-left module,
     \item $\forall \omega \in \module{P}, \forall d \in \D',\quad \CDN(d\omega)\leq (\Lip'(d)+\norm{d}{\D'})\CDN(\omega)$,
     \item for all $j\in\{1,2\}$, the pair $(\pi^j, \Theta^j)$ is a left module morphism from the left $\D'$-module $\module{P}$ to the left $\A^j$-module $\module{M}_j$.
   \end{enumerate}
\end{definition}

\begin{definition}[{\cite{Latremoliere18d}}]
  The \emph{extent}, $\tunnelextent{\tau,\tau'}$, of a metrical tunnel $(\tau,\tau')$ is given by
    \begin{equation*}
      \tunnelextent{\tau,\tau'} = \max\left\{\tunnelextent{\tau},\tunnelextent{\tau'}\right\} \text{.}
    \end{equation*}
\end{definition}

\begin{definition}[{\cite{Latremoliere18d}}]\label{metrical-prop-def}
  The \emph{metrical propinquity}, $\dmetpropinquity{}(\mathds{A},\mathds{B})$, between two {\gQVB s} $\mathds{A}$ and $\mathds{B}$ is the nonnegative number given by
  \begin{equation*}
    \dmetpropinquity{}(\mathds{A},\mathds{B}) = \inf\left\{\tunnelextent{\tau} : \text{ $\tau$ is a metrical tunnel from $\mathds{A}$ to $\mathds{B}$} \right\}\text{.}
  \end{equation*}
\end{definition}

\begin{theorem}[{\cite{Latremoliere18d}}]
  The metrical propinquity is a complete metric, up to full isometry, on the class of metrical quantum vector bundles. 
\end{theorem}

We are now able to prove the following useful result.

\begin{theorem}\label{metrical-conv-thm}
  The following limit holds:
  \begin{multline*}
    \lim_{n\rightarrow\infty} \dmetpropinquity{}\big( \big(\Hilbert_n,\CDN_n,\C,0,C(\SG{n}), \Lip_n\big), \\ \big(\Hilbert_\infty,\CDN_\infty,\C,0,C(\SG{\infty}), \Lip_\infty\big) \big) = 0 \text{.}
  \end{multline*}
\end{theorem}

\begin{proof}
  We use the framework and notation of the proof of Theorem \ref{prop-conv-thm} and Lemma \ref{modular-conv-lemma} --- in particular, for all $n\in\N$, the D-norm $\mathsf{T}_n$ on $\module{M} = \Hilbert_\infty\oplus\Hilbert_n$ and the associated modular tunnel $\mu_n$. We also use the tunnels $\tau_{n,\alpha}$ from $(C(\FC{n}),\Lip_n)$ to $(C(\FC{\infty}),\Lip_\infty)$, built by using the L-seminorms $\mathsf{M}_{n,\alpha}$.
  
  Let $\varepsilon > 0$ be given, and let $N\in\N$ be such that for all $n\geq N$,
  \begin{itemize}
  \item $\tunnelextent{\mu_n} < \varepsilon$,
  \item $\Haus{d_n}(\FC{n},V_n) < \frac{\varepsilon}{4}$ and
  \item $\Haus{d_\infty}(\FC{\infty},V_n) < \frac{\varepsilon}{4}$.
  \end{itemize}
  
  Using the alternate proof of Theorem \ref{prop-conv-thm} given toward the end of Section 2, we see that the extent of the tunnel $\tau_{n,\frac{\varepsilon}{4}} = (\A_n,\mathsf{M}_{n,\frac{\varepsilon}{4}},\rho_n,\rho_\infty)$ is at most equal to $\frac{3 \varepsilon}{4}$. We let $\tau_n$ be the tunnel $\tau_{n,\frac{\varepsilon}{4}}$.

  \bigskip
  
  Let $f \in \sa{C(\FC{\infty})}$ and $g \in \sa{C(\FC{n})}$. Let $\xi \in \Hilbert_\infty$ and $\eta\in\Hilbert_n$. We compute successively:
\begin{align*}
  \mathsf{T}_n(f\xi, g\eta)
  &= \max\left\{ \CDN_\infty(f\xi), \CDN_n(g\eta), \frac{1}{\varepsilon}\norm{f \xi - g \eta}{\Hilbert_\infty} \right\} \\
  &\leq \max\left\{
    \begin{array}{l}
      (\norm{f}{C(\FC{\infty})} + \Lip_\infty(f)) \CDN_\infty(\xi) \text{,}\\
      (\norm{g}{C(\FC{n})} + \Lip_n(g)) \CDN_n(\eta) \text{,} \\
      \frac{1}{\varepsilon}\norm{f\xi - g\eta}{\Hilbert_\infty}
    \end{array} \right\}\\
  &\leq \max\left\{
    \begin{array}{l}
      (\norm{(f,g)}{\A_n} + \mathsf{M}_{n,\frac{\varepsilon}{4}} (f,g)) C(\xi,\eta) \text{,}\\
      \frac{1}{\varepsilon}\norm{f\xi - g\eta}{\Hilbert_\infty}
    \end{array} \right\}\text{.}
\end{align*}

Now, by construction, since $\eta \in \Hilbert_n$, then $f\eta$, meant to stand for $\pi_\infty(f)\eta$, is well defined (as $\Hilbert_n$ is a subspace of $\Hilbert_\infty$) and moreover, $\pi_\infty(f)\eta = \pi_n(f_n)\eta$ (written as $f\eta = f_n \eta$), where $f_n$ is the restriction of $f$ to $\FC{n}$, by construction of $\pi_n$ and $\pi_\infty$.

We thus compute successively:
\begin{align*}
  \frac{1}{\varepsilon}\norm{f\xi - g\eta}{\Hilbert_\infty}
  &\leq \frac{1}{\varepsilon}\left(\norm{f(\xi-\eta)}{\Hilbert_\infty} + \norm{f\eta-g\eta}{\Hilbert_\infty}\right)\\
  &\leq \frac{1}{\varepsilon}\left(\norm{f(\xi-\eta)}{\Hilbert_\infty} + \norm{f_n\eta-g\eta}{\Hilbert_n}\right)\\
  &\leq \frac{1}{\varepsilon}\left(\norm{f(\xi-\eta)}{\Hilbert_\infty} + \norm{(f_n-g)\eta}{\Hilbert_n}\right)\\
  &\leq \norm{f}{C(\FC{\infty})} \mathsf{T}_n(\xi,\eta) + \frac{1}{\varepsilon}\norm{f_n-g}{C(\FC{n})} \norm{\eta}{\Hilbert_n} \text{.}
\end{align*}

If $x\in \FC{n}$, then there exists $v \in V_n$ such that $d_\infty(x,v) < \frac{1}{4}\varepsilon$ and $d_n(x,v) < \frac{1}{4}\varepsilon$, and therefore, we obtain the following upper estimate:
\begin{align*}
  |f_n(x)-g(x)|
  &\leq |f_n(v)-g(v)| + |f_n(x)-f(v)| + |g(x) - g(v)| \\
  &= |f(v)-g(v)| + |f(x)-f(v)| + |g(x) - g(v)| \\
  &\leq \frac{\varepsilon}{4} \left(\mathsf{M}_{n,\frac{\varepsilon}{4}}(f,g) + \Lip_\infty(f) + \Lip_n(g)\right)\\
  &\leq \varepsilon \mathsf{M}_{n,\frac{\varepsilon}{4}}(f,g) \text{.}
\end{align*}

Thus,
\begin{equation*}
  \frac{1}{\varepsilon}\norm{f\xi-g\eta}{\Hilbert_\infty} \leq \left(\norm{f}{C(\FC{\infty})} + \mathsf{M}_{n,\frac{\varepsilon}{4}}(f,g)\right) \mathsf{T}_n(\xi,\eta)\text{.}
\end{equation*}

In conclusion, we have shown that
\begin{equation*}
  \mathsf{T}_n(f\xi,g\eta) \leq \left( \norm{(f,g)}{\A_n} + \mathsf{M}_{n,\frac{\varepsilon}{4}}(f,g) \right) \mathsf{T}_n(\xi,\eta) \text{,}
\end{equation*}
as desired.

Therefore, we have constructed a metrical tunnel whose extent $\chi$ satisfies the following estimate:
\begin{equation*}
  \xi = \max\{\tunnelextent{\mu_n}, \tunnelextent{\tau_{n,\frac{\varepsilon}{4}}}\} \leq \varepsilon\text{.}
\end{equation*}
This concludes the proof of the theorem.
\end{proof}

Let us review our situation. In the previous section, we easily observed that $(\FC{n},d_n)_{n\in\N}$ converges in the Gromov--Hausdorff distance to $(\FC{\infty},d_\infty)$. Moreover, we proved this convergence within the functional analytic framework of the propinquity, in preparation for the study of the convergence of spectral triples over these spaces.

Spectral triples give rise to certain kind of modules over {\qcms s}, called {\gQVB s} --- these modules are Hilbert spaces (i.e., Hilbert modules over $\C$) endowed with the graph norms of the Dirac operators. We proved in Lemma \ref{modular-conv-lemma} that, indeed, these modules converge in the sense of the modular propinquity, which is an extension of the construction of the propinquity to {\gQVB s}, available to us thanks to the functional analytic framework.

In fact, spectral triples also include the action of a {\qcms} on their associated {\gQVB s}. The resulting objects, dubbed metrical quantum vector bundles, also converge, per Theorem \ref{metrical-conv-thm}, for yet one more extension of the modular propinquity to metrical quantum vector bundles. Each of these extensions of the original propinquity consists in following the model laid out by the construction of the propinquity $\dpropinquity{}$, but also requires more complex diagrams in order to define tunnels. Interestingly, the only quantity we ever need to compute are distances between {\qcms s} --- the rigidity of the required diagrams in the definition of tunnels is sufficient to define well--behaved metrics.

Now, the following natural question arises. Given any two spectral triples, we could apply the definition of the metrical propinquity to their associated metrical vector bundles, as we did here. This defines a pseudo-metric on spectral triples, which we temporarily write as
\begin{equation*}
  \widetilde{\Lambda}((\A_1,\Hilbert_1,D_1),(\A_2,\Hilbert_2,D_2)) = \dmetpropinquity{}(\qvb{\A_1}{\Hilbert_1}{D_1},\qvb{\A_2}{\Hilbert_2}{D_2}) \text,
\end{equation*}
between any two metric spectral triples $(\A_1,\Hilbert_1,D_1)$ and $(\A_2,\Hilbert_2,D_2)$; see the comment after Definition \ref{propinquity-def} for the definition of a pseudo-metric. Is it a metric, up to a natural equivalence? In \cite{Latremoliere18g}, the third author proved that, indeed,
\begin{equation*}
  \widetilde{\Lambda}((\A_1,\Hilbert_1,D_1),(\A_2,\Hilbert_2,D_2)) = 0
\end{equation*}
if and only if there exists a unitary operator $U$ from $\Hilbert_1$ to $\Hilbert_2$ which conjugates the *-representations of $\A_1$ and $\A_2$ on, respectively, $\Hilbert_1$ and $\Hilbert_2$, and such that
\begin{equation}\label{u-d1-d2-eq}
  U^\ast |D_1| U = |D_2| \text{,}
\end{equation}
where, for all $j\in\{1,2\}$, the operator $|D_j|$ is the absolute value of the operator $D_j$, defined via the continuous functional calculus for (possibly unbounded) self-adjoint operators. However, we see that more information must be captured if we want to recover the spectral triples up to a unitary equivalence. We would like to replace Equation (\ref{u-d1-d2-eq}) with the stronger expression $U^\ast D_1 U = D_2$. This is the subject of the definition of the spectral propinquity, which builds itself on top of the metrical propinquity by including covariant quantities. This is the subject of the next subsection.

\subsection{Convergence of Spectral Triples}

If $(\A,\Hilbert,D)$ is a spectral triple, then $D$ induces a strongly continuous action of $\R$ on $\Hilbert$ by unitaries, defined by
\begin{equation*}
  \forall t \in \R, \quad U(t) = \exp(i t D)\text{,}
\end{equation*}
using the functional calculus applied to the self--adjoint operator $D$.

Given two spectral triples, we have discussed in the previous subsection how to make sense of their module structures being close; the last point to address is how to measure how close the actions of $\R$ they induce are.

The covariant version of the propinquity was introduced in \cite{Latremoliere18b}. The covariant modular propinquity was then introduced in \cite{Latremoliere18g}, and it is the tool needed to define the spectral propinquity. A full description of the covariant modular propinquity is not necessary for our purposes; instead, we restrict our attention to the sort of metrical bundles obtained from spectral triples and the actions of $\R$ described above in Equation (\ref{qvb-def-eq}).

Let us assume that we are given two metric spectral triples $(\A_1,\Hilbert_1,D_1)$ and $(\A_2,\Hilbert_2,D_2)$. In Equation (\ref{qvb-def}), we have defined the metrical quantum vector bundles $\qvb{\A_1}{\Hilbert_1}{D_1}$ and $\qvb{\A_2}{\Hilbert_2}{D_2}$, associated respectively with the metric spectral triples $(\A_1,\Hilbert_1,D_1)$ and $(\A_2,\Hilbert_2,D_2)$. Let $\mathds{P} = (\tau,\tau')$ be a metrical tunnel between $\qvb{\A_1}{\Hilbert_1}{D_1}$ and $\qvb{\A_2}{\Hilbert_2}{D_2}$, such as the ones built in Theorem \ref{metrical-conv-thm}.

    In particular, let us write $\tau = (\mathds{P},(\Phi_1,\phi_1),(\Phi_2,\phi_2))$ and note that by Definition \ref{metrical-tunnel-def} for metrical tunnels, $\tau$ is a modular tunnel  between $(\Hilbert_1,\CDN_1,\C,0)$ and $(\Hilbert_2,\CDN_2,\C,0)$, where $\CDN_1$ and $\CDN_2$ are, respectively, the graph norms of $D_1$ and $D_2$. Furthermore, let us write $\mathds{P} = (\mathscr{P},\CDN_\D,\D,\Lip_\D)$.

    Since $D_1$ and $D_2$ are self-adjoint, we can define two strongly continuous actions of $\R$ by unitaries on $\Hilbert_1$ and $\Hilbert_2$ by letting
    \begin{equation*}
      \forall j \in \{1,2\}, \quad \forall t \in \R, \quad U_j^t = \exp(i t D_j) \text{.}
    \end{equation*}

    The \emph{spectral propinquity} is defined by extending the metrical propinquity of Definition \ref{metrical-prop-def} to include the actions $U_1$ and $U_2$.
    
    With this in mind, let $\varepsilon > 0$ and assume that $\tunnelextent{\tau} \leq \varepsilon$. Let us call a pair $(\varsigma_1,\varsigma_2)$ of maps from $\R$ to $\R$ an $\varepsilon$-\emph{iso-iso}, for some $\varepsilon > 0$, whenever
    \begin{equation*}
      \forall j \in \{1,2\}, \quad \forall x,y,z \in \left[-\frac{1}{\varepsilon},\frac{1}{\varepsilon}\right], \quad \left| |\varsigma_j(x) + \varsigma_j(y)-z| - |(x + y) - \varsigma_k(z)| \right| \leq \varepsilon \text{,}
    \end{equation*}
    and $\varsigma_1(0) = \varsigma_2(0) = 0$.
    
    As discussed in \cite{Latremoliere18b}, such maps can be used to define a distance on the class of proper monoids, but for our purpose, as we only work with the proper group $\R$, the definition simplifies somewhat. In fact, we only recall the definition so that we may properly define the spectral propinquity below: for our purposes, the only iso-iso map we will work with is simply the identity of $\R$.

    Thus, suppose that we are given an $\varepsilon$-iso-iso $(\varsigma_1,\varsigma_2)$ from $\R$ to $\R$, as above. We call $(\tau,\tau',\varsigma_1,\varsigma_2)$ an $\varepsilon$-covariant metrical tunnel.

The \emph{$\varepsilon$-covariant reach} of $(\tau,\varsigma_1,\varsigma_2)$ is then defined as follows:
\begin{multline*}
  \max_{j \in \{1,2\}} \sup_{\substack{\xi \in \Hilbert^j\\ \CDN^j(\xi)\leq 1}}\inf_{\substack{\xi'\in\Hilbert^k\\ \CDN^k(\xi')\leq 1}}\sup_{|t| \leq \frac{1}{\varepsilon}} \\
  \sup_{\substack{\omega\in\module{P}\\\CDN_\D(\omega)\leq 1}} \left|\inner{U^j(t)\xi}{\Pi_j(\omega)}{\Hilbert^j} - \inner{U^k(\varsigma^k(t))\xi'}{\Pi_k(\omega)}{\Hilbert^k}\right| \text{.}
\end{multline*}

We define the \emph{$\varepsilon$-magnitude} $\tunnelmagnitude{\tau,\tau',\varsigma_1,\varsigma_2}{\varepsilon}$ of an $\varepsilon$-covariant tunnel $(\tau,\varsigma_1,\varsigma_2)$ to be the maximum of the extent of $\tau$, the extent of $\tau'$, and the $\varepsilon$-covariant reach of $(\tau,\varsigma_1,\varsigma_2)$. 

The \emph{spectral propinquity}
\begin{equation*}
  \spectralpropinquity{}((\A_1,\Hilbert^1,D^1),(\A_2,\Hilbert^2,D^2))
\end{equation*}
is the nonnegative number
\begin{equation*}
  \min\left\{\frac{\sqrt{2}}{2},\inf\{\varepsilon>0 : \exists \text{ $\varepsilon$-covariant metrical tunnel $\tau$, } \quad \tunnelmagnitude{\tau}{\varepsilon} \leq \varepsilon \right\} \text{.}
\end{equation*}

We refer to \cite{Latremoliere18b,Latremoliere18d} for a discussion of the fundamental properties of this metric, as a special case of the covariant metrical propinquity, including a discussion of sufficient conditions for completeness (on certain classes).

\bigskip

We record the following property of this metric.
\begin{theorem}[{\cite{Latremoliere18g}}]
  The spectral propinquity $\spectralpropinquity{}$ is a metric on the class of metric spectral triples, up to the following coincidence property: for any metric spectral triples $(\A_1,\Hilbert^1,D^1)$ and $(\A^2,\Hilbert^2,D^2)$,
  \begin{equation*}
    \spectralpropinquity{}((\A_1,\Hilbert^1,D^1),(\A_2,\Hilbert^2,D^2)) = 0
  \end{equation*}
  if and only if there exists a unitary map $V : \Hilbert^1\rightarrow \Hilbert^2$ such that
  \begin{equation*}
    V D_1 V^\ast = D_2
  \end{equation*}
  and
  \begin{equation*}
    \mathrm{Ad}_V = V (\cdot) V^\ast \text{ is a *-isomorphism from $\A_1$ onto $\A_2$},
  \end{equation*}
  where (as is customary) we identify $\A_1$ and $\A_2$ with their images by their representations on $\Hilbert^1$ and $\Hilbert^2$, respectively. In particular, the *-isomorphism from $\A^1$ onto $\A^2$ implemented by the adjoint action of $V$ is a full quantum isometry from $(\A^1,\opnorm{[D^1,\cdot])}{}{\Hilbert^1}$ onto $(\A^2,\opnorm{[D^2,\cdot]}{}{\Hilbert^2})$.
\end{theorem}

We are now able to conclude this section with the formal statement, and the proof, of our main result in this paper.

\begin{theorem}
  If Hypothesis \ref{main-hyp} holds, then
  \begin{equation*}
    \lim_{n\rightarrow\infty} \spectralpropinquity{}\left(\left(C(\FC{n}),\Hilbert_n,D_n\right),\left(C(\FC{\infty}),\Hilbert_\infty,D_\infty\right)\right) = 0 \text{.}
  \end{equation*}
\end{theorem}

\begin{proof}
We use in this proof the notation of Theorem \ref{prop-conv-thm}, Lemma \ref{modular-conv-lemma} and Theorem \ref{metrical-conv-thm}.

For each $t \in \R$ and $n\in\Nbar$, we set
\begin{equation*}
  U_n(t) = \exp( i D_n ),
\end{equation*}
viewed as an operator on $\Hilbert_n$, where $\Hilbert_n$ is seen as a subspace of $\Hilbert_\infty$.

Let $\xi \in \Hilbert_n$. we note that $D_n\xi = D_\infty\xi$, so that $U_n(t)\xi = U_\infty(t)\xi$, for all $t\in \R$. 

With this in mind, we see that if $(\xi,\eta)\in\Hilbert_n\oplus\Hilbert_\infty$ with $C_n(\xi,\eta)\leq 1$, then, in particular, $\norm{\xi-\eta}{\Hilbert_\infty} \leq \varepsilon$. Therefore, we obtain the following estimate, valid for all $t\in \R$:
\begin{align*}
  \norm{U_n(t) \xi - U_\infty(t)\eta}{\Hilbert_\infty}
  &\leq \norm{U_n(t)\xi-U_\infty(t)\xi}{\Hilbert_\infty} + \norm{U_\infty(t)\xi-U_\infty(t)\eta}{\Hilbert_\infty} \\
  &\leq \norm{U_n(t)\xi-U_\infty(t)\xi}{\Hilbert_\infty}  + \norm{\xi-\eta}{\Hilbert_\infty} \\
  &\leq 0 + \varepsilon = \varepsilon \text{.}
\end{align*}

Since, obviously, $(\mathrm{id}_\R,\mathrm{id}_\R)$ is a $0$-iso-iso, the above computation shows, by the Cauchy--Schwarz inequality, that the metrical tunnel $(\mu_n,\tau_n)$ has covariant reach at most $\varepsilon$.

Since the extent of $(\mu_n,\tau_n)$ is also at most $\varepsilon$, we conclude, as desired, that
\begin{equation*}
  \spectralpropinquity((C(\FC{n}),\Hilbert_n,D_n),(C(\FC{\infty}),\Hilbert_\infty,D_\infty)) \leq \varepsilon \text{.}
\end{equation*}
This concludes the proof of our main theorem.
\end{proof}

\section{Future Directions}

We conclude this paper with several suggestions for future directions of research, based upon the results and methods presented here:
\begin{itemize}\setlength{\itemsep}{5pt}
\item The metrical propinquity is known to be complete, and so the conditions for completeness of the covariant propinquity are discussed in \cite{Latremoliere18c}. It is therefore natural to construct spectral triples on fractals sets described as limits of simpler sets by using the completeness of the propinquity. Thus, the methods described here provide a test case for this broader goal.

\item The propinquity (in its different versions) is a tool from noncommutative geometry and is defined on classes of noncommutative C*-algebras and their modules. It therefore opens up the possibility of discussing the notion of a \emph{noncommutative fractal}. To this end, a closed class of {\qcms s}, or even a complete class of metric spectral triples, could be endowed with a finite collection of maps which act as contractions for the corresponding propinquity. We would then  typically use the completeness once more in order to obtain a limit, by the Picard fixed point theorem, which will be noncommutative if all the elements in our class are. A very simple such example can be obtained by working with matrix-valued functions over compact subsets of the plane --- it is easy to define L-seminorms for these spaces so that, for any fixed $d\in\N\setminus\{0\}$, the noncommutative C*-algebra of $d\times d$ matrix-valued continuous functions on the {\SiepG} is indeed the limit, for the propinquity, of the sequence of C*-algebras of $d\times d$ matrix-valued continuous functions over the approximating sets $\SG{n}$. It would be very interesting to find and investigate richer examples, including noncommutative fractals constructed on a simple C*-algebra.

\item Future progress on the properties of the spectral propinquity can be applied to the current work in order to devise new methods to compute certain quantities attached to fractals and related to spectral triples, such as gleaning new information about the associated spectral zeta functions, and connections between convergence for the propinquity and spectral complex dimensions (see Remark \ref{m-rmk} above for references on complex dimensions, especially \cite{seventeenhalf, eighteenhalf}). Moreover, we will discuss in a future work how the present framework can be used to discuss the convergence of the energy forms on the pre--fractal approximations to the {\SiepG} and to the harmonic gasket.

\item The primary purpose of the third author in introducing the propinquity and its variants was to construct quantum field theories by constructing algebras of observables by approximations, using the completeness of the propinquity. A fascinating possibility to be explored is the construction of a mathematically rigorous quantum field theory over the {\SiepG} and other piecewise $C^1$-fractal curves.

\item One of the second author's key motivations in \cite{thirtyfourhalf,Lapidus94,Lapidus97} was to eventually be able to define and study a suitable notion of a ``fractal manifold'' and of a ``fractal (as well as noncommutative) space-time'', along with the associated moduli spaces, by analogy with general relativity and the approach to quantum physics via Feynman path integrals \cite{twentyonehalf}. (See also \cite{sixtytwohalf}.) Examples of such fractal manifolds should include piecewise $C^1$-fractal curves and, in particular, the {\SiepG} and the harmonic gasket. The techniques and results developed and obtained in this work, together, in particular, with the work in \cite{Lapidus94,Lapidus97,Lapidus08,Lapidus14} and \cite{Kigami08,ninehalf}, should help realize this long-term goal.
  
\end{itemize}

\providecommand{\bysame}{\leavevmode\hbox to3em{\hrulefill}\thinspace}
\providecommand{\MR}{\relax\ifhmode\unskip\space\fi MR }
\providecommand{\MRhref}[2]{%
  \href{http://www.ams.org/mathscinet-getitem?mr=#1}{#2}
}
\providecommand{\href}[2]{#2}

\vfill

\end{document}